\documentclass[11pt,twoside,reqno,centertags,english]{amsart}
\usepackage[utf8]{inputenc}
\usepackage{babel}
\usepackage{lmodern}
\usepackage{bbding}
\usepackage{hyperref}
\usepackage{geometry}
\usepackage{comment}
\hypersetup{
colorlinks,
    citecolor=blue,
    filecolor=blue,
    linkcolor=blue,
    urlcolor=blue
    }
\def\be{\begin{equation}}
\def\ee{\end{equation}}
\newcommand {\p} {\partial}
\usepackage{amsmath, amsfonts, amssymb}
\usepackage{amsthm}
\newtheorem{theorem}{Theorem}[section]
\newtheorem{lemma}[theorem]{Lemma}
\newtheorem{assumption}[theorem]{Assumption}

\newtheorem{corollary}[theorem]{Corollary}
\newtheorem{proposition}[theorem]{Proposition}
\theoremstyle{remark}
\newtheorem{remark}[theorem]{Remark}

\theoremstyle{definition}
\newtheorem{definition}[theorem]{Definition}

\usepackage{mathtools}

\makeatletter
\DeclareRobustCommand\widecheck[1]{{\mathpalette\@widecheck{#1}}}
\def\@widecheck#1#2{%
    \setbox\z@\hbox{\m@th$#1#2$}%
    \setbox\tw@\hbox{\m@th$#1%
       \widehat{%
          \vrule\@width\z@\@height\ht\z@
          \vrule\@height\z@\@width\wd\z@}$}%
    \dp\tw@-\ht\z@
    \@tempdima\ht\z@ \advance\@tempdima2\ht\tw@ \divide\@tempdima\thr@@
    \setbox\tw@\hbox{%
       \raise\@tempdima\hbox{\scalebox{1}[-1]{\lower\@tempdima\box
\tw@}}}%
    {\ooalign{\box\tw@ \cr \box\z@}}}
\makeatother

\DeclareMathOperator{\grad}{grad}

\DeclareMathOperator{\tr}{Tr}

\DeclareMathOperator{\di}{div}
\newcommand{\vertiii}[1]{{\left\vert\kern-0.25ex\left\vert\kern-0.25ex\left\vert #1 
		\right\vert\kern-0.25ex\right\vert\kern-0.25ex\right\vert}}
\newcommand{\bigk}{\mathcal K(E,B)}
\newcommand{\bigkp}{\mathcal K(E_+,B_+)}
\newcommand{\bigkm}{\mathcal K(E_m,B_m)}

\newcommand {\R} {\mathbb{R}}

\newcommand{\iu}{\mathrm{i}} 

\newcommand{\je}{\mathfrak J}
\newcommand{\jj}{\tilde{\mathcal J}}
\newcommand{\nn}{N\times N}

\numberwithin{equation}{section}
\newcommand{\Wei}{\mathfrak{H}}

\newcommand{\wei}{\mathfrak{h}}
\newcommand{\lkt}{L_{\tilde\ce}}
\newcommand{\lks}{L_{\ce^*}}
\newcommand{\lk}{L_{\ce}}
\newcommand{\lkss}{L_{\ce_s}}

\newcommand{\mom}{v^{\#}}
\newcommand{\mou}{u^{\#}}

\newcommand{\Lam}{\Lambda}
\date{}
\newcommand{\leql}{\preceq}
\newcommand{\geql}{\succeq}
\newcommand{\Rea}{\mathfrak{Re}}
\newcommand{\Ima}{\mathfrak{Im}}
\newcommand{\ce}{\mathbf{K}}

\newcommand{\aaa}{\delta}
\newcommand{\bbb}{\beta}

\newcommand{\cf}{\mathbf{F}}
\newcommand{\cb}{\mathfrak{D}}
\newcommand{\qq}{Q}

\newcommand{\lamproj}{\pi_\Lam}
\newcommand{\Mmm}{\R^{\nn}_+}
\newcommand{\Mm}{\R^{\nn}_\Lam}
\title[Dafermos' principle]{Dafermos' principle and Brenier's duality scheme for defocusing dispersive equations}
\author[D.~Vorotnikov]{Dmitry Vorotnikov}
\address[D.~Vorotnikov]{University of Coimbra, CMUC, Department of Mathematics,  3001-501 Coimbra, Portugal}{}
\email{mitvorot@mat.uc.pt}

\begin{document}

\begin{abstract} We discover an abstract structure behind several nonlinear dispersive equations (including the  NLS,  NLKG and GKdV equations with generic defocusing power-law nonlinearities) that is reminiscent of hyperbolic conservation laws. The underlying abstract problem admits an ``entropy'' that is formally conserved.  The entropy is determined by a strictly convex function that naturally generates an anisotropic Orlicz space. For such problems, we introduce the dual matrix-valued variational formulation in the spirit of \emph{Y. Brenier. Comm. Math. Phys. (2018) 364(2) 579-605}. Employing time-adaptive weights, we are able to prove consistency of the duality scheme on large time intervals. We also prove solvability of the dual problem in the corresponding anisotropic Orlicz spaces.  As an application, we show that no subsolution of the PDEs that fit into our framework is able to dissipate the total entropy earlier or faster than the strong solution on the interval of existence of the latter.   This result (we call it Dafermos' principle) is new even for ``isotropic'' problems such as the incompressible Euler system. 
\end{abstract}
\maketitle

Keywords: nonlinear PDE, convex duality,  entropy, anisotropic Orlicz space, generalized optimal transport, Dafermos' criterion

\textbf{MSC [2020] 35D99, 35L90, 37K58, 47A56, 49Q99}

\section{Introduction}  
Cauchy problems for nonlinear evolutionary PDEs can have infinitely many weak solutions, cf. \cite{DP79,lel09,lel10,BNF16}. Many of such problems possess a physically relevant quantity (a Hamiltonian, energy or entropy depending on the context) that should be formally conserved along the flow, but can fail to do it for the weak solutions. There is some consensus that for physically relevant solutions this quantity (for definiteness, hereafter we refer to it as the \emph{``total entropy''}) cannot exceed its initial value, cf. \cite{DS17}. No universal selection criterion for weak solutions has been found, but in physics, information theory, chemistry and biology there exist 
numerous variational principles that can potentially assist. The most famous ones are Prigogine’s principle also known as the minimum entropy production principle \cite{Prigozhin} that is applicable to open systems and  Ziegler’s principle also known as the maximum entropy production principle \cite{Ziegler} that is suitable for closed systems. Other notable principles are due to  Onsager, Gyarmati, Berthelot, Swenson, Lotka, Enskog, Kohler, Haken, Paltridge, Malkus, Veronis and Jaynes, see \cite{MS06} for a survey. Such principles can be employed both for derivation of physically relevant systems of PDEs and
for selection of physically relevant solutions \cite{glimm}. In this vein, Dafermos \cite{Daf1,Daf2} suggested to assume that the physically relevant weak solutions dissipate the total entropy earlier and faster than the irrelevant ones. Typically, this kind of entropy\footnote{The wording comes from the theory of conservation laws and might be confusing for some readers. For example, smooth solutions of the heat equation dissipate the Boltzmann entropy, so this entropy is never conserved. In this paper, as we have already agreed, the ``total entropy'' is a formally conserved quantity.} is conserved for smooth solutions but can dissipate for weak solutions due to shocks etc. It is important that the criterion acts locally near the ``bifurcation'' moment of time, i.e.,  before a certain moment the total entropies of a ``good'' and a ``bad'' solution are equal (this stage is optional and the discrepancy can already occur at the initial moment of time), but shortly after that moment the total entropy of a ``bad'' solution becomes larger than the total entropy of a ``good'' solution; however, it is permitted that, as time elapses, the total entropy of a ``bad'' solution returns to being smaller than or equal to the one of a ``good'' solution. 

Appropriateness of Dafermos' principle was examined for various PDEs, see, e.g., \cite{Hsiao,Daf2,Fei1,CK14,CJ16}.  Dafermos' criterion has recently been applied \cite{GK,CA23} at the (wider) level of \emph{subsolutions} to some problems of fluid dynamics. Numerical applications of the criterion  have recently been produced in \cite{Klei22,Klei23,Klei24}. 

For systems of hyperbolic conservation laws
\begin{equation}\label{conslaw} \p_t v= -\p_x (\cf(v)),\end{equation} where $\cf:\R^n \to \R^n$ is a given flux function, there exist classical selection criteria (Lax’ shock condition, the Kruzhkov-Lax entropy criterion, the vanishing viscosity criterion and some others) that under some assumptions are compatible with Dafermos' principle, see \cite{Daf1,KS97}. For equations of Lagrangian and continuum mechanics, a certain least action admissibility criterion has recently been advocated in \cite{gi2024}. A related least action principle was discussed in \cite{GHK}.

For the incompressible Euler and some related equations, Brenier \cite{CMP18,Br20} suggested to search for the solution that minimizes the time average of the kinetic energy. This problem might not always admit a solution, but it generates a dual variational problem that has better convexity features. He found an explicit formula that relates the smooth solutions of the incompressible Euler existing on a small time interval to the solutions of the corresponding dual problem. In \cite{BMO22}, this technique was applied to the multi-stream Euler-Poisson system. In \cite{V22}, we developed Brenier's approach by finding structures in nonlinear quadratic PDEs that permit to define similar dual problems with decent properties. It was also discussed in \cite{CMP18,Br20,V22} that dual problems of this kind have strong analogy with the optimal transportation problems. 

A similar duality scheme has been (rather independently) proposed by Acharya and several collaborators, see \cite{Ach4,Ach6,Ach5} and the references therein. The recent preprint \cite{Ach7} builds upon both approaches. Among other results, the authors of \cite{Ach7} prove the $\Gamma$-convergence of the dual problem for the incompressible Navier-Stokes system to the corresponding dual problem for the incompressible Euler. From the perspective of the analogy with the optimal transport, this result is reminiscent of the $\Gamma$-convergence results from \cite{Leonard12,BaradatMonsaingeon18,MTV23}. A notable feature of \cite{Ach7} is that their construction of the dual problem involves a ``base state'' that can be regarded as an ``initial guess'' for the solution of the primal problem. (From this point of view, the ``base states'' in \cite{CMP18,Br20,BMO22,V22} are identically zero).

It was clear from the very beginning \cite{CMP18,Ach4} that the nonlinearity in a PDE does not need to be quadratic in order to implement the construction. However, the majority of rigorous results for this kind of problems has until now been obtained for quadratic nonlinearities. Of course, for many relevant systems the nonlinearity fails to be quadratic, for instance, for multidimensional systems of conservation laws 
\be \label{e:conslaw} \p_t v=- \di (\cf(v)),\ee where $\cf$ is a prescribed matrix function, and for various nonlinear dispersive equations. Remarkably, rigorous study of Brenier's duality scheme for multidimensional systems of conservation laws \eqref{e:conslaw} that admit a convex entropy, initiated in \cite{CMP18}, is still pending. 

In this paper, we extend Brenier's approach to systems of PDEs that can be rewritten in the abstract form \be \label{o:aeulerint}\p_t v= L(\cf(v)),\ee where the matrix function $\cf$  has  some convexity and positivity properties to be specified below in Assumptions \ref{convf} and \ref{assl}, and $L$ is a vector-valued differential operator of any order. To fix the ideas, we will restrict ourselves to unknown vector functions $v(t,x)$ with $x$ varying in the periodic box $\Omega=\mathbb T^d$. Note that  \eqref{o:aeulerint} is strongly reminiscent of \eqref{conslaw} and \eqref{e:conslaw}. We will also need to assume that \eqref{o:aeulerint} admits a strictly convex anisotropic \emph{entropy} function $\ce(v)$ that is formally conserved along the flow as described in Assumptions \ref{ass1} and \ref{a:consc}. As we will see, the  NLS,  NLKG and GKdV equations with generic defocusing power-law nonlinearities can be recast in our abstract form. 

As an application of our duality scheme, we rigorously prove the following variant of Dafermos' principle: no subsolution of \eqref{o:aeulerint} can dissipate the total entropy earlier or faster than the strong solution on the interval of existence of the latter (Theorem \ref{t:daf}). Our Dafermos' principle is new even for quadratic problems that fit into our framework, including the incompressible Euler equation \cite{CMP18}, the ideal incompressible MHD, the multidimensional Camassa-Holm, the Zakharov-Kuznetsov system and many others (various examples can be found in \cite{V22}).

Our entropy $\ce$ generates an anisotropic Orlicz space, hence our procedure will heavily rely on the theory of anisotropic Orlicz spaces, see, e.g., \cite{Gwiazda}. In particular, we will prove existence of solutions to the dual problem that belong to an anisotropic Orlicz space. Notably, some authors employed isotropic Orlicz spaces for the study of nonlinear Schr\"odinger and generalized KdV equations, cf. \cite{hofer2024}, but we have never heard of any application of anisotropic Orlicz spaces --- that appear so naturally in our generic approach --- to this sort of PDEs. 

Until recently, an unsettling detail of the duality scheme was that the majority of consistency results --- that advocate the possibility to view solutions to the dual problems as certain ``dual variational solutions`'' to PDEs that admit the considered duality approach --- were only obtained on small time intervals, cf. \cite[Theorems 2.3 and 3.1]{CMP18}, \cite[Theorem 3.8]{V22}, \cite[Theorem 5.1]{BMO22}. In this paper, we are able to guarantee the consistency on large time intervals by introducing appropriate time-dependent weights to the problem. Global in time consistency theorems for the dual formulations of incompressible Euler and Navier-Stokes equations have recently been obtained in \cite{Ach7}. The latter results heavily rely on a suitable choice of the ``base state''. Nevertheless, our consistency results  seem to be of a different nature, and, furthermore, the two attitudes complement each other to a certain degree, see Remark \ref{remcons} for a thorough comparison.

 We believe that it is potentially possible to prove the weak-strong uniqueness property in the sense that existence of a (suitably defined) strong solution to \eqref{o:aeulerint} implies that the solution to the dual problem is unique in a certain anisotropic Orlicz space. The corresponding result in the quadratic case $\cf(v)=v\otimes v$ was established in \cite[Section 5]{V22}. As the first step towards this conjecture, we will show that the strong solution is always unique. The proof is unusual because it relies on the analysis of a specific ``Jeffreys divergence''.
 
 To the best of our knowledge, global in time solvability for the considered dispersive PDEs is merely known under significant restrictions on the exponents in the power laws or on the size of the initial data, cf. \cite{sulem2007nonlinear,LP14,FRR23}. Our Theorem \ref{t:ex} and Remark \ref{r:time}  provide existence of certain ``dual variational solutions" for these problems without such restrictions. 

The paper consists of seven sections (including the Introduction) and two appendices. In Section \ref{orl}, we list some  necessary facts about $N$-functions and anisotropic Orlicz spaces.  In Section \ref{s:abst}, we describe our abstract formalism. In Section \ref{Secx}, we establish the consistency of the dual problem, and explain how it can be secured on large time intervals. As an application, we prove the already mentioned version of Dafermos' principle. In Section \ref{Secsol}, we show that the dual problem always has a solution that belongs to a specific anisotropic Orlicz space. In Section \ref{Secun}, we obtain uniqueness of strong solutions. In Section \ref{Sec3}, we demonstrate how the  NLS,  NLKG and GKdV equations with generic defocusing power-law nonlinearities (as well as scalar conservation laws) fit into our framework. In Appendix \ref{la1}, we prove a technical anisotropic variant of Hardy's inequality (we do not claim much novelty here, but we failed to find it in the existing literature).  In Appendix \ref{bot}, we employ an unsophisticated geometric point of view in order to illustrate the analogy of the dual problems in question with the optimal transport theory.

\section{Anisotropic Orlicz spaces} \label{orl} In this section we collect some basic necessary facts about $N$-functions and anisotropic Orlicz spaces.  
 
 \begin{definition}[Young function] A convex function $\Phi:\R\to \R_+$ is called a \emph{Young function} provided $\Phi$ is even, $\Phi(s)=0$ if and only if $s=0$, and $$\lim_{s\to 0} \frac {\Phi(s)} s=0,\ \lim_{s\to +\infty} \frac {\Phi(s)} s=+\infty.$$

 	\end{definition}
 	
 	Note that the Legendre transform $\Phi^*$ of a Young function is well defined and is also a Young function. 
 	
 \begin{definition}[$N$-function] A convex function $\ce:\R^n\to \R_+$ is called an \emph{$N$-function} provided $\ce$ is even, and there exist two Young functions $\Phi_-$ and $\Phi_+$ such that $$\Phi_-(|y|)\le \ce (y) \le \Phi_+(|y|), \ y\in \R^n.$$

 \end{definition}
 
 It follows from the definition that if $\ce$ is an $N$-function, then 
 
 (i) $\ce(y)=0$ if and only if $y=0$;
 
 (ii) The Legendre transform $\ce^*$ is well-defined and is also an $N$-function. 
 
 \begin{definition}[$\Delta_2$-condition] An $N$-function $\ce:\R^n\to \R_+$ satisfies the \emph{$\Delta_2$-condition} provided\footnote{Hereafter, we write $b \lesssim a$ for two scalar expressions $a$ and $b$ to indicate that $b \leq C a$ with a uniform constant $C$.  } $$\ce (2y) \lesssim \ce(y)+1, \ y\in \R^n.$$

 \end{definition}
 
\begin{assumption} Henceforth in this section, we fix an $N$-function $\ce$. We always assume that both $\ce$ and $\ce^*$ satisfy the  $\Delta_2$-condition, and all the facts are claimed under this assumption. \end{assumption}

Let $\qq$ be either an interval, or a flat torus (periodic box), or the Cartesian product of an interval and a torus, equipped with the Lebesgue measure $dx$.
\begin{definition}[Modular and anisotropic Orlicz class] The integral $$\int_\qq \ce(f(x))\,dx$$ is called the \emph{modular} of a measurable function $f:\qq\to \R^n$. Denote by $\lk(\qq;\R^n)$  the set of functions $f$ whose modular is finite. \end{definition}
Note that  $\lk(\qq;\R^n)$  is a linear space, cf. \cite[Lemma 3.3.1]{Gwiazda}, and it is called an \emph{anisotropic Orlicz space}. For $\ce(v)=\frac {|v|^p}{p}$, $p>1$, the spaces $\lk(\qq;\R^n)$ are classical Lebesgue spaces $L^p(\qq;\R^n)$ . 

There are two different ways to define a relevant norm on $\lk(\qq;\R^n)$.
\begin{definition}[Luxemburg and Orlicz norms] The \emph{Luxemburg norm} of $f\in \lk(\qq;\R^n)$ is defined as $$\|f\|_{\lk(\qq)}:=\inf\left\{\lambda>0: \int_{\qq}\ce\left(\frac {f(x)}{\lambda}\right)\,dx\leq 1\right\}.$$ 
The \emph{Orlicz norm} of $f\in \lk(\qq;\R^n)$ is defined as $$\vertiii{f}_{\lk(\qq)}:=\sup \left\{\int_{\qq}f(x)\cdot g(x)\,dx: \int_{\qq}\ce^*\left(g(x)\right)\,dx\leq 1\right\}.$$ \end{definition} Both norms are indeed norms. Moreover, $\lk(\qq;\R^n)$ equipped with any of these norms is a separable reflexive Banach space, see \cite{Gwiazda}. Both norms are equivalent with uniform constants: \be \|f\|_{\lk(\qq)}\leq \vertiii{f}_{\lk(\qq)} \leq 2\|f\|_{\lk(\qq)}.\ee Furthermore, the strong topology associated with both norms coincides with the \emph{modular topology} \cite{Gwiazda}, i.e., 
\be \|f_k\|_{\lk(\qq)}\to 0\ \Leftrightarrow\ \int_\qq \ce(f_k(x))\,dx\to 0\ee as $k\to +\infty$.
\begin{proposition}[H\"older inequality \cite{Gwiazda}] For every $f\in \lk(\qq;\R^n)$ and $g\in \lks(\qq;\R^n)$ one has $$\left| \int_{\qq}f(x)\cdot g(x)\,dx\right|\le 2  \|f\|_{\lk(\qq)} \|g\|_{\lks(\qq)}.$$
\end{proposition}
Another useful estimate, \cite[Lemma 3.1.14]{Gwiazda}, is \be \label{e:modcont} \|f\|_{\lk(\qq)} \leq \max\left\{1, \int_\qq \ce(f(x))\,dx\right\}.\ee  
\begin{proposition}[Duality \cite{Gwiazda}] Bounded linear functionals on $\lk(\qq;\R^n)$ can be represented in the form $$f \mapsto \int_{\qq}f(x)\cdot g(x)\,dx$$ for some $g\in  \lks(\qq;\R^n)$. Consequently, $\lks(\qq;\R^n)$ is the dual space of $\lk(\qq;\R^n)$ .
\end{proposition}
We will need the following (refined) embedding of anisotropic Orlicz spaces.  \begin{proposition} Let $\ce, \tilde \ce$ be two $N$-functions. Assume that there exists $r>0$ such that $\ce (v) \leq \tilde \ce (v)$ for $|v|\geq r$. Then for any $f\in \lkt(\qq;\R^n)$ we have $f\in \lk(\qq;\R^n)$ and $$\|f\|_{\lk(\qq)}\le C_r \|f\|_{\lkt(\qq)}$$ where $C_r$ merely depends on $r$ and $\qq$.     \label{la0} \end{proposition}
The proof of Proposition \ref{la0} is omitted because the classical proof for the isotropic case, cf. \cite[Theorem 13.3]{KR}, is applicable here as well, mutatis mutandis.  
\begin{proposition}[Density of smooth functions] Smooth functions are dense in $\lk(\qq;\R^n)$. 
\end{proposition}
\begin{proof} This fact is well known, so we just give a sketch of the proof.  It follows from \cite[Lemma 2.3.16]{Gwiazda} that there exist $r>0$ and $p>1$ such that $\ce(v)\leq \frac {|v|^p} p$ for $|v|\geq r$. By Proposition \ref{la0}, $L^p(\qq)\subset \lk(\qq)$ and $\|\cdot\|_{\lk(\qq)}\le C_r \|\cdot\|_{L^p(\qq)}$. Simple functions are dense in $\lk(\qq;\R^n)$ by \cite[Corollary 3.4.12]{Gwiazda}. On the other hand, in any $L^p$-neighborhood (and hence in any $L_\ce$-neighborhood) of a simple function there is a smooth function. \end{proof}

We will also need the following variant of Krasnoselskii's theorem about continuity of Nemytskii operators. 
\begin{proposition} \label{p:nemyt} Let $\eta:\R^n\to \R^m$ be a continuous function such that the Nemytskii operator $$\mathcal N: \lk(\qq;\R^n) \to L^1(\qq;\R^m),\ \mathcal N(w)(x)=\eta(w(x))$$ is well-defined. Then $\mathcal N$ is continuous. 
\end{proposition} We omit the proof because  in the original proof of Krasnoselskii, see \cite[Chapter 1, Theorem 2.1]{Kold}, it is possible to substitute the anisotropic Orlicz source $\lk$  for the Lebesgue source $L^p$ without destroying the validity of the argument\footnote{The case of anisotropic targets is more tricky and lies beyond this discussion; for instance, the proof for isotropic Orlicz sources and targets that can be found in \cite{KR} does not work in the anisotropic case.}.

We finish this section with a statement about mollification. We restrict ourselves to the case when $\qq$ is a $d$-dimensional periodic box. Let $\rho$ be a smooth compactly supported probability density on $\R^d$. Set $\rho_k(x):=k\rho(kx)$, $k\in \mathbb N$. For any fixed $f\in \lk(\qq;\R^n)$, the convolution $\rho_k*f$ is well-defined, smooth, and belongs to $\lk(\qq;\R^n)$. 

\begin{proposition} \label{p:conv} We have $\rho_k*f\to f$ strongly in $\lk(\qq;\R^n)$ as $k\to +\infty$. \end{proposition}

\begin{proof} We give just a brief sketch of the proof. It follows from the density of smooth functions that the ``translation'' map $$x\mapsto f(\cdot-x)$$ is well-defined and continuous from $\R^d$ to $\lk(\qq;\R^n)$. We conclude by mimicking the ``isotropic'' Donaldson-Trudinger argument from \cite[Lemma 2.1]{DT}.  \end{proof}

\section{The abstract setting} \label{s:abst}

Let us start with fixing some basic notation. In the what follows, $\Omega$ is the periodic box $\mathbb T^d$, $d\in \mathbb N$, equipped with the Lebesgue measure $dx$. For two measurable functions $u,v:\Omega\to \R^m$, $m\in \mathbb N$, we will use the shortcut $$(u,v):=\int_\Omega u(x)\cdot v(x)\,dx.$$ 
We will also use the notations $\R^{\nn}$, $\R^{\nn}_s$ and $\Mmm$ for the spaces of $N\times N$ matrices, symmetric matrices and nonnegative-definite matrices, resp. For $A,B\in \R^{\nn}_s$, we write $A\geql B$ when $A-B\in \Mmm$ (the Loewner order \cite{L30}; we stress that the notation $A\ge B$ is reserved for a less restrictive order that will introduced in Assumption \ref{assl} below), and $A:B$ for the scalar product of $A$ and $B$ generated by the Frobenius norm $|\cdot|$.  The symbol $I$ will stand for the identity matrix of a relevant size.

As anticipated in the Introduction, we are interested in the problem \be \label{o:aeuler}\p_t v= L(\cf(v)), \ v(0,\cdot)=v_0.\ee Here $v_0:\Omega\to \R^n$, $n\in \mathbb N$, is the initial datum, $v:[0,T]\times \Omega\to \R^n$ is an unknown vector function, $$\cf:\R^n\to \R^{\nn}_s$$ is a prescribed $C^2$-smooth matrix function with some convexity and positivity properties to be specified below, and $L$ is a vector-valued differential operator\footnote{As a matter of fact, $L$ can be a more general (and possibly nonlocal) linear operator; it is easy to see from the considerations below that it is enough to assume that $L$ commutes with the mollification $L(\rho_k*M)=\rho_k*L(M)$ and is a closed linear operator with a dense domain $D(L)$ in a certain space.} with constant coefficients $$L(\Xi)_i=\sum_{j=0}^{\nu} \sum_{l,m=1}^{N}  \sum_{|\alpha|=j}  b_{ilm\alpha} \p_\alpha \Xi_{lm},\ i=1,\dots,n,$$ where $\alpha$ is a generic multiindex and $\Xi(x)\in \R^{\nn}_s$. 


\begin{assumption}[Entropy] \label{ass1} Define the ``entropy'' function for our problem by $$\ce: \R^n \to \R,\ \ce(v):=\frac 1 2 \tr (\cf(v)-\cf(0)).$$ Assume that $\ce$  is strictly convex (but not necessarily uniformly convex). Assume also that $\ce$ is an $N$-function, and both $\ce$ and its Legendre transform $\ce^*$ satisfy the $\Delta_2$-condition. \end{assumption}

To any given $v\in \R^n$ we associate the ``sharp'' vector $$\mom:=\nabla \ce(v),$$ where the gradient $\nabla$ is naturally taken w.r.t. $v$. Note that $0^{\#}=0$ because $\nabla \ce (0)=0$.  Since $\ce$  is strictly convex, the map $$\nabla \ce: v \mapsto \mom$$ is $C^1$-smooth and injective. Moreover, $$v=\nabla \ce^*(\mom)$$ for any $\mom\in \nabla \ce(\R^n)=\R^n$.  

It follows from \cite[Lemma 2.3.16]{Gwiazda} that there exist $p,q>1$ and $r>0$ such that $$\frac {v\cdot  \mom} {\ce(v)}\leq p,\quad  \frac {\mom \cdot  v} {\ce^*(\mom)}\geq q$$ for $|v|\ge r$.  Consequently, \be \label{e:lham} \ce^*(\mom)\lesssim \ce (v)+1, \ v\in \R^n.\ee Similarly, \be \label{e:lham2} \ce (v)\lesssim \ce^*(\mom)+1, \ v\in \R^n.\ee
 
\begin{assumption}[$\Lam$-convexity] \label{convf} We will require some convexity of the matrix function $\cf$.  The classical notion of convexity of matrix-valued functions is the Loewner convexity \cite{KR36,BW}, namely, $\cf:\R^n\to \R^{\nn}_s$ is Loewner convex if $\cf:P$ is a  convex function for any matrix $P\in  \Mmm$. This is too restrictive for our purposes, cf. Remark \ref{lowtoostrong} and Appendix \ref{bot}, so we will merely assume that $\cf$ is \emph{$\Lam$-convex} in the following more relaxed sense. Let $u:\Omega \to \R^n$ be an arbitrary smooth function, and let $L^*$ be the differential operator adjoint to $L$. Let $\Lam\subset \R_s^{\nn}$ be the smallest linear subspace of $\R_s^{\nn}$ independent of $u$ and containing $I$ such that $L^*(u(x))\in \Lam$, $x\in \Omega$.  We say that $\cf$ is $\Lam$-convex if $\cf:P$ is a  convex function for any $P\in \Lam \cap \Mmm$. Of course, any Loewner convex matrix function is $\Lam$-convex, but not vice versa. Notably, it follows from the $\Lam$-convexity of $\cf$ that $\ce$ is convex.  \end{assumption}

\begin{assumption}[$\Lam$-order] \label{assl} We will also require some positivity of the matrix function $\cf$. For this purpose, we will use the less restrictive $\Lam$-order instead of the Loewner order. Namely, denote $$\Mm=\{\Xi\in \R_s^{\nn}|\ \Xi:P\geq 0, \ \forall P \in  \Lam \cap \Mmm\}.$$ We will assume that \be \cf(\R^n)\subset \Mm.\ee For $A,B\in \R^{\nn}_s$, we write $B\le A$ when $A-B\in \Mm$.  Obviously, $B \leql A$ implies $B \le A$, and the converse is true only if $\Lam=\R_s^{\nn}$. Employing some simple algebra, it is easy to derive from $\cf\ge 0$ that
\be\label{e:lessss0} |\lamproj \cf(v)| \lesssim \ce(v)+1,\ee where $\lamproj:\R_s^{\nn}\to \Lam$ denotes the orthogonal projection w.r.t. the Frobenius product. We will also need to assume that pointwise control similar to \eqref{e:lessss0} holds at the level of partial derivatives $$\cf_l(v):=\cf_{v_l}(v), \l=1,\dots,n.$$  More exactly, we assume that
\be\label{e:lessss} | \lamproj  \partial_l \cf(v)| \lesssim |\nabla \ce(v)| +1.\ee 
\end{assumption}

\begin{remark} \label{mos} In some examples, including the incompressible Euler equation,  cf. \cite{CMP18,V22}, $\Lam$ coincides with $\R_s^{\nn}$, hence in this case the $\Lam$-order and the $\Lam$-convexity coincide with their Loewner counterparts. For the dispersive equations considered in Section \ref{Sec3}, the subspace $\Lam$ is significantly smaller than $\R_s^{\nn}$. \end{remark}

\begin{remark} The ``flux'' $\cf(v)$ can be viewed as an ``entropy matrix''  that very vaguely resembles the entropy matrix from \cite{Shen24}. Notably, the definition of convexity of matrix-valued functions employed in \cite{Shen24} is very restrictive, it is much stronger than the Loewner convexity. \end{remark}

\begin{remark} The quadratic flux $\cf(v)=v\otimes v$ and the corresponding entropy $\ce(v)=\frac 1 2 |v|^2$ obviously satisfy Assumptions \ref{ass1}, \ref{convf} and \ref{assl} for any subspace $\Lam\subset \R_s^{n\times n}$. Note that in this case $N=n$ and $\mom=v$. This applies to the incompressible Euler and other quadratic examples mentioned in the Introduction. \end{remark}

\begin{assumption}[Conservativity] \label{a:consc} We will focus on the situation when $L$ satisfies the ``formal conservativity condition'' \be \label{o:acons} (\cf(v), L^*(\mom))=0\ee provided $\mom:\Omega\to \R^n$ is a smooth function (remember that $v=\nabla \ce^*(\mom)$). \end{assumption}

Problem \eqref{o:aeuler} admits the following natural weak formulation: \be \label{o:w1}\int_0^T \left[(v,\p_t a)+(\cf(v), L^*a)\right]\, dt+(v_0, a(0,\cdot))=0\ee for all smooth vector fields $a: [0,T]\times \Omega\to \R^n$, $a(T,\cdot)\equiv 0$.

\begin{remark}[Formulation in terms of the sharp variable] In the conservative case \eqref{o:aeuler} admits a formally equivalent formulation in terms of the ``sharp'' unknown variable $\mom$: \be \label{o:aeuler2}\p_t (\mom)_l+L^* (\mom):\partial_l \cf(\nabla \ce^*(\mom))=0, \quad \mom(0,\cdot)=\mom_0:=(v_0)^\#,\ l=1,\dots,n.\ee  Indeed, let $a: [0,T]\times \Omega \to \R^n$ be a smooth test function, $s\in \R$ be a real number, and $v$ be a solution to \eqref{o:aeuler}. Then due to the conservativity \be  (\cf(v+sa), L^*(\nabla \ce(v+sa)))=0.\ee Taking the derivative w.r.t. $s$ at $s=0$ yields \be \label{o:aeulepp} \sum_{l=1}^n \left[(\p_l\cf(v)a_l, L^*(\mom))+(L(\cf(v)),\p_{v_l}(\nabla \ce (v))a_l)\right]=0.\ee
Using \eqref{o:aeuler}, we recast \eqref{o:aeulepp} in the form 
\be \label{o:aeulepp1} \sum_{l=1}^n (\p_l\cf(v):L^*(\mom),a_l)+\sum_{l,m=1}^n(\p_t v_m,\p_{lm}\ce (v)a_l)=0.\ee
By the chain rule, $\p_t (\mom)_l=\p_t (\p_l \ce(v))=\sum_{m=1}^n \p_{lm}\ce (v)\p_t v_m$. Thus \eqref{o:aeulepp1} becomes \be \label{o:aeulepp2} \sum_{l=1}^n (\p_l\cf(v):L^*(\mom)+\p_t (\mom)_l,a_l)=0,\ee which immediately implies \eqref{o:aeuler2} due to arbitrariness of $a$ and the relation $v=\nabla \ce^*(\mom)$. \end{remark}


Let us now rewrite problem \eqref{o:w1} in terms of the test functions $B:=L^*a$ and $E:=\p_t a$ (note that the conservativity \eqref{o:acons} is not needed at this stage). We first observe that \be \label{o:vab} (v_0, a(0))=-\int_0^T \left(v_0,   \p_t a\right)\, dt=-\int_0^T \left(v_0,   E\right)\, dt. \ee  The link between $B$ and $E$ can alternatively be described by the conditions \be\label{o:constr1}\p_t B=L^* E, \quad B(T,\cdot)\equiv 0.\ee 

Hence, \eqref{o:w1} becomes \be \label{o:w2}\int_0^T \left[(v-v_0,E)+(\cf(v), B)\right]\, dt=0\ee for all smooth vector fields $B: [0,T]\times \Omega\to \R^{\nn}_s$, $E: [0,T]\times \Omega\to \R^n$ satisfying the constraints \eqref{o:constr1}.

Observe that \eqref{o:constr1} can be rewritten in the following weak form \be \label{o:constrweak}\int_0^T \left[(B,\p_t \Psi)+(E, L \Psi)\right]\, dt=0\ee for all smooth matrix fields $\Psi: [0,T]\times \Omega\to \R^{\nn}_s$, $\Psi(0)= 0$.

Motivated by the discussion above, we adopt the following definitions, where we tacitly assume  $v_0\in \lk(\Omega;\R^n)$. 
\begin{definition}[Weak solutions] A function \be \label{e:vclass} v\in \lk((0,T)\times \Omega;\R^n)\ee  is a \emph{weak solution} to \eqref{o:aeuler} if it satisfies \eqref{o:w2} for all pairs \be \label{o:be} (E,B)\in \lks ((0,T)\times \Omega;\R^n)\times L^\infty((0,T)\times \Omega;\R^{\nn}_s)\ee meeting the constraint \eqref{o:constrweak}. \end{definition}

Note that if $v$ is a weak solution, then $\cf(v)\in L^1((0,T)\times \Omega;\R^{\nn}_s)$ due to \eqref{e:lessss0}, so \eqref{o:w2} makes sense. 

\begin{definition}[Subsolutions] \label{d:ss} A pair of functions  \be(v,M)\in \lk((0,T)\times \Omega;\R^n)\times L^1((0,T)\times \Omega;\R^{\nn}_s), \ \cf(v)\leq M, \label{e:lorder} \ee is a \emph{subsolution} to \eqref{o:aeuler} if it satisfies \be \label{o:w2sub}\int_0^T \left[(v-v_0,E)+(M, B)\right]\, dt=0\ee for all pairs $$ (E,B)\in \lks ((0,T)\times \Omega;\R^n)\times L^\infty((0,T)\times \Omega;\R^{\nn}_s)$$ meeting the constraint \eqref{o:constrweak}. \end{definition}
The second inequality  in \eqref{e:lorder} is understood in the sense of the $\Lam$-order a.e. in $(0,T)\times \Omega$. Obviously, if $v$ is a weak solution, then $(v,\cf(v))$ is a subsolution. Accordingly, the subsolution \emph{entropy} function  that complies with Assumption \ref{ass1} is $$M\mapsto \frac 1 2 \tr (M-\cf(0)),\ M\in \Mm.$$

Observe that for a weak solution $v$ the modular (that we will call the \emph{total entropy}) \begin{equation} \label{o:consp1} K(t):=\int_\Omega \ce (v(t,x))\,d x.\end{equation} belongs to $L^1(0,T)$ and thus is finite for almost all times. Similarly, the total entropy \begin{equation} \tilde K(t):=\frac 1 2 \tr\int_\Omega \left[M(t,x)-\cf(0)\right]\,d x=\frac 1 2 (M(t,\cdot)-\cf(0),I).\end{equation} of a subsolution also belongs to $L^1(0,T)$. 

\begin{remark} For the incompressible Euler equation, in view of Remark \ref{mos} and some observations made in \cite{DS17,V22}, Definition \ref{d:ss} is equivalent  to the conventional definition of a subsolution  that, among other applications, is used in the theory of convex integration \cite{lel09,lel10}. \end{remark}

Our primal problem is to search for the weak solutions to \eqref{o:aeuler} that minimize a suitable  weighted time integral of the total entropy\footnote{See also Remark \ref{r:zar}.} \eqref{o:consp1}. 

More exactly, fix a positive scalar function $\wei(t)$ that is bounded away from $0$ and $\infty$ on $[0,T]$. Let $\Wei(t):=\int_t^T \wei(s)\, ds$. The idea is to search for the weak solutions that minimize $\int_0^T \wei(t) K(t)\, dt$. 

\begin{remark}[``Rough'' Dafermos' criterion] A typical weight is $\wei(t):=\exp(-\gamma t)$  with large $\gamma>0$, cf. Remark \ref{r:timeei}. In this case our selection criterion can be viewed as a ``rough'' Dafermos' principle,  because we prioritize the solutions that dissipate the total entropy as fast and as early as possible and that at later times do not admit dramatic exponential growth (of order $\exp(\gamma t)$) of the total entropy. \end{remark}

We will now employ the "sharp" formulation \eqref{o:aeuler2} in order to define \emph{strong solutions} to \eqref{o:aeuler}.

\begin{definition}[Strong solutions] \label{d:strongclass} Assume that $v_0\in \lk(\Omega;\R^n)$. A function $v$  satisfying \eqref{e:vclass} and \be\label{o:strongclass} \p_t (\Wei\mom)\in \lks ((0,T)\times \Omega;\R^n), \ \Wei L^* (\mom)\in L^\infty ((0,T)\times \Omega; \R^{\nn}_s) \ee  is a \emph{strong solution} to \eqref{o:aeuler} if it is a weak solution and \be \label{o:aeuler22}\p_t (\mom)_l+L^* (\mom):\partial_l \cf(\nabla \ce^*(\mom))=0, \quad \mom(0,\cdot)=\mom_0\ee  a.e. in $(0,T)\times \Omega$ and in $\Omega$, resp. \end{definition}

\begin{remark} \label{r:regul1} Let us explain why \eqref{o:aeuler22} makes sense for $v$ from the regularity class \eqref{e:vclass}, \eqref{o:strongclass}. Firstly, \eqref{e:lham} and \eqref{e:vclass} imply $\mom \in \lks((0,T)\times \Omega; \R^n)$. The first equality in \eqref{o:aeuler22} is equivalent to \be \label{o:aeuler2w}\wei (\mom)_l+\p_t (\Wei\mom)_l+\Wei L^* (\mom):\partial_l \cf(\nabla \ce^*(\mom))=0\ee a.e. in $(0,T)\times \Omega$. By \eqref{e:lessss}, $\lamproj \partial_l \cf(v)$ is pointwise controlled by $\nabla \ce(v)=\mom$, therefore all members of \eqref{o:aeuler2w} belong to $\lks((0,T)\times \Omega; \R^n)$. We moreover claim that \be \mom \in C([0,T);\lks(\Omega; \R^n))\ee and that all members in the second equality in \eqref{o:aeuler22} belong to $\lks(\Omega; \R^n)$. Indeed, since $\ce^*(\mom_0)\lesssim \ce (v_0)+1$, we have $\mom_0\in \lks(\Omega; \R^n)$. But $$\p_t \mom=\frac 1 \Wei \p_t (\Wei\mom)+\frac \wei \Wei \mom\in L_{\ce^*, loc}([0,T)\times \Omega;\R^n)\subset L^{1}_{loc}([0,T); L^1(\Omega)),$$ whence $$ \mom \subset AC_{loc}([0,T); L^1(\Omega)).$$ Leveraging convexity of $\ce^*$, for every $t\in [0,T)$ and small $h>0$ we obtain  \begin{multline}\int_\Omega \ce^*(\mom(t+h,x)- \mom(t,x))\,dx= \int_\Omega \ce^*\left(\frac 1 h \int_{t}^{t+h} h \p_t \mom(s,x)\,ds\right)\,dx\\ \le  \int_\Omega \frac 1 h \int_{t}^{t+h} \ce^*\left( h \p_t \mom(s,x)\right)\,ds\,dx\\=\int_\Omega \frac 1 h \int_{t}^{t+h} \ce^*\left( h \p_t \mom(s,x)+(1-h)0\right)\,ds\,dx\\ \le \int_\Omega \frac 1 h \int_{t}^{t+h} \left(h\ce^*\left(\p_t \mom(s,x)\right)+(1-h)\ce^*\left(0\right)\right)\,ds\,dx\\=\int_\Omega \int_{t}^{t+h} \ce^*\left(\p_t \mom(s,x)\right)\,ds\,dx\to 0\end{multline} as $h\to 0$ due to absolute continuity of the Lebesgue integral. This secures the right-continuity of the map  $\mom: [0,T)\to \lks(\Omega)$; the left-continuity is proved in  a similar fashion. \end{remark}

\begin{remark} \label{r:equiv} Definition \ref{d:strongclass} is actually independent of the weight function $\Wei$. Indeed, let $v$ be a strong solution, $\wei_1$ be another positive scalar function bounded away from $0$ and $\infty$, and $\Wei_1=\int_t^T \wei_1(s)\, ds$. Obviously, $\Wei(t) \lesssim  \Wei_1(t)\lesssim \Wei(t)$. Hence, $$\Wei_1 L^* (\mom)\in L^\infty ((0,T)\times \Omega; \R^{\nn}_s),$$ and  $$ \p_t (\Wei_1\mom)=-\wei_1 \mom+ \Wei_1 \p_t (\mom)=-\wei_1 \mom+\frac {\Wei_1} {\Wei}\wei \mom+\frac {\Wei_1} {\Wei} \p_t (\Wei\mom).$$ Since $\lks ((0,T)\times \Omega;\R^n)$ is a linear space invariant w.r.t. multiplication by bounded functions, $$\p_t (\Wei_1\mom)\in \lks ((0,T)\times \Omega;\R^n).$$  
\end{remark}

The following lemma shows that the strong solutions soar above the rest of the weak solutions. 

\begin{lemma} \label{lemm} If \eqref{o:acons} is assumed, then for any strong solution $v$ the total entropy is continuous in time for $t\in [0,T)$ and is conserved, i.e.,  $K(t)=K(0)=\int_\Omega \ce (v_0)\,d x$.  \end{lemma}

\begin{proof} Let $\varphi(t)$ be an arbitrary smooth compactly supported function on $(0,T)$. Observe that $\frac {\varphi}{\Wei}$ is bounded, Lipschitz and compactly supported on $(0,T)$. Let us prove that \be \label{e:cl11}\int_0^T(\ce(v), \p_t\varphi)\, dt=0.\ee We first claim that \be \label{e:w3++a}\int_0^T(\cf(v), L^*(\varphi\mom))\, dt=0\ee  and \be \label{e:w3++b}\int_0^T \left[(v-v_0,\p_t(\varphi \mom))-(\ce(\nabla \ce^*(\mom)), \p_t \varphi)\right]\, dt=0.\ee Taking \eqref{e:w3++a} and \eqref{e:w3++b} for granted, and employing  \eqref{o:w2} with $$B=L^*(\varphi \mom)=\frac {\varphi}{\Wei}  \Wei L^*(\mom)\in  L^\infty ((0,T)\times \Omega;\R^n),$$ $$ E=\p_t(\varphi\mom)=\p_t( \frac {\varphi}{\Wei})\Wei\mom+\p_t (\Wei\mom) \frac {\varphi}{\Wei}\in  \lks ((0,T)\times \Omega;\R^n),$$ we infer $$\int_0^T (v-v_0,\p_t(\varphi\mom))\, dt=0$$ and, consequently, $$\int_0^T(\ce(\nabla\ce^*(\mom)), \p_t\varphi)\, dt=0.$$ 
	
	We have proved \eqref{e:cl11}, i.e., that the total entropy is conserved in the sense of distributions. In order to prove that it is conserved in the classical sense, we just need to establish the continuity of the total entropy w.r.t. time. By Remark \ref{r:regul1}, $\mom \in C([0,T);\lks(\Omega; \R^n))$. Proposition \ref{p:nemyt} and \eqref{e:lham2} imply that the Nemytskii operator $$\mathcal N_1: w \mapsto \ce(\nabla \ce^*(w))$$ is continuous from $\lks (\Omega;\R^n)$ to $L^1(\Omega)$. Consequently, $\ce(v) \in C([0,T);L^1(\Omega))$, whence $K(t)\in C([0,T))$. 
	
	It remains to prove \eqref{e:w3++a} and \eqref{e:w3++b}. Without loss of generality (continuing by zero outside $[0,T]$) we may assume that $\mom$ and $\varphi$ are defined on the torus $\qq:=2T\mathbb T^1\times \Omega$. Consider the mollifications $\mom_k:=\rho_k * \mom$, where $\rho_k$ are as in Section \ref{orl}. 
	
We now prove \eqref{e:w3++b}. Due to the presence of the cut-off function $\varphi$, the behaviour of $\mom$ for small $T-t$ does not affect the validity of \eqref{e:w3++b}, and thus without loss of generality we may assume that $$\p_t \mom=\frac 1 \Wei \p_t (\Wei\mom)+\frac \wei \Wei \mom\in \lks (\qq;\R^n).$$ Then\be  \label{e:lpcon}  \mom_k\to \mom,\ \p_t \mom_k \to \p_t \mom, \ \p_t (\varphi\mom_k) \to \p_t (\varphi\mom) \ee strongly in $\lks (\qq;\R^n)$ as $k\to +\infty$. Set $v_k:=\nabla \ce^*(\mom_k)$.  Observe now that \be \label{f223}\int_0^T \left[(v_k,\p_t(\varphi \mom_k))-(\ce(\nabla \ce^*(\mom_k)), \p_t \varphi)\right]\, dt=0.\ee Indeed, integration by parts gives \begin{multline} \int_0^T \left[(v_k,\p_t(\varphi \mom_k))-(\ce(v_k), \p_t \varphi)\right]\, dt\\=- \int_0^T \left[(\p_t v_k,\varphi \mom_k)-(\nabla\ce(v_k)\cdot \p_t v_k,  \varphi)\right]\, dt\\= - \int_0^T \left[(\p_t v_k,\varphi \mom_k)-(\mom_k\cdot \p_t v_k,  \varphi)\right]\, dt= 0.\end{multline} 

By \eqref{e:lham2}, \be \label{e:hamst}\int_\qq \ce (\nabla \ce^*(w))\lesssim \int_\qq \ce^* (w)+1.\ee Proposition \ref{p:nemyt} implies that the Nemytskii operator $$\mathcal N_1: w \mapsto \ce(\nabla \ce^*(w))$$ is continuous from $\lks (\qq;\R^n)$ to $L^1(\qq)$. Hence, we can pass to the limit in the second term of \eqref{f223}.  

On the other hand,  it follows from \eqref{e:lpcon} and \eqref{e:hamst} that the sequence $v_k=\nabla \ce^*(\mom_k)$ is bounded in $\lk(\qq;\R^n)$ and a.e. converges to $v$. Passing to a subsequence if necessary, we infer $v_k\to v$ weakly in $\lk(\qq;\R^n)$. We can now legitimately pass to the limit in the first term of \eqref{f223}.  

Moreover,
\be \label{e:w3++d}\int_0^T (v_0,\p_t(\varphi \mom_k))\, dt=0\ee since $\varphi$ is compactly supported in $(0,T)$ and $v_0$ does not depend on $t$.  Since $v_0\in \lk (\Omega;\R^n)\subset \lk (\qq;\R^n)$, we can pass to the limit in \eqref{e:w3++d} as well. 

As a result, \eqref{f223} and \eqref{e:w3++d} imply \eqref{e:w3++b}.
	
	 Let us prove \eqref{e:w3++a}. Due to the presence of $\varphi$, the behaviour of $\mom$ for small $T-t$ is again not relevant, and we may assume that $$L^* (\mom)\in L^\infty (\qq; \R^{\nn}_s).$$ Passing to a subsequence if necessary, we get \be  \label{e:lpcon1} L^* (\mom_k) \to L^* (\mom)\ee weakly-$*$ in $L^\infty (\qq; \R^{\nn}_s)$ as $k\to +\infty$. On the other hand, \eqref{o:acons} yields \be \label{f228} \int_0^T(\cf(\nabla \ce^*(\mom_k)), L^*(\varphi\mom_k))\, dt=\int_0^T(\cf(\nabla \ce^*(\mom_k)), L^*(\mom_k))\varphi(t)\, dt=0.\ee  Estimates \eqref{e:lessss0}, \eqref{e:hamst} and Proposition \ref{p:nemyt} imply that the Nemytskii operator $$\mathcal N_2: w \mapsto \cf(\nabla \ce^*(w))$$ is continuous from $\lks (\qq;\R^n)$ to $L^1(\qq;\R^{\nn}_s)$. Hence, $$\cf(v_k)=\cf(\nabla \ce^*(\mom_k))\to\cf(\nabla \ce^*(\mom))=\cf(v)$$ strongly in $L^1(\qq;\R^{\nn}_s)$, and passing to the limit in \eqref{f228} we obtain \eqref{e:w3++a}.
\end{proof}

The suggested approach of selecting the weak solutions that minimize $\int_0^T \wei(t) K(t)\, dt$ --- or, equivalently, $ \frac 1 2  \int_{(0,T)\times \Omega} \wei(t) \tr \cf(t)\,dx\,dt$ --- can be implemented via the saddle-point problem \be\label{o:sadd1}\mathcal I(v_0, T)=\inf_{v}\sup_{E,B:\,\eqref{o:constrweak}}\int_0^T \left[(v-v_0,E)+\frac 1 2(\cf(v), \wei I+2B)\right]\,dt.\ee The infimum in \eqref{o:sadd1} is taken over all $v\in \lk((0,T)\times \Omega;\R^n)$, and the supremum is taken over all pairs $(E,B)$ satisfying \eqref{o:be} and the linear constraint \eqref{o:constrweak}.

The dual problem is \be\label{o:sadd2}\mathcal J(v_0, T)=\sup_{E,B:\,\eqref{o:constrweak}}\inf_{v}\int_0^T \left[(v-v_0,E)+\frac 1 2(\cf(v), \wei I+2B)\right]\,dt,\ee where $v,E,B$ are varying in the same function spaces as above. 

It is also relevant to consider the corresponding problems for the subsolutions: \be\label{o:sadd1sub}\tilde{\mathcal I}(v_0, T)=\inf_{v,M:\cf(v)\leq M}\sup_{E,B:\,\eqref{o:constrweak}}\int_0^T \left[(v-v_0,E)+\frac 1 2(M, \wei I+2B)\right]\,dt.\ee The infimum in \eqref{o:sadd1} is taken over all $v\in \lk((0,T)\times \Omega;\R^n)$ and $M\in L^1((0,T)\times \Omega;\Mm)$, and the supremum is taken over all pairs $(E,B)$ satisfying \eqref{o:be} and the linear constraint \eqref{o:constrweak}.
The problem dual to \eqref{o:sadd1sub} is \be\label{o:sadd2sub}\tilde{\mathcal J}(v_0, T)=\sup_{E,B:\,\eqref{o:constrweak}}\inf_{v,M:\cf(v)\leq M}\int_0^T \left[(v-v_0,E)+\frac 1 2(M, \wei I+2B)\right]\,dt,\ee where $v,M,E,B$ are varying in the same function spaces as above. 

\begin{remark}[Simple observations about the optimal values] Denote \be K_0:=\frac 1 2 \int_\Omega \tr \cf(v_0)\, dx=K(0)+\frac 1 2 \int_\Omega \tr \cf(0)\, dx.\ee Since $\inf\sup\geq \sup\inf$, one has $$\mathcal I(v_0, T)\ge \tilde{\mathcal I}(v_0, T)\ge \tilde{\mathcal J}(v_0, T), \ \mathcal I(v_0, T)\ge \mathcal J(v_0, T)\ge \tilde{\mathcal J}(v_0, T).$$ Note that the $\sup$ in \eqref{o:sadd1} is always $+\infty$ if $v$ is not a weak solution, whence $\mathcal I(v_0,T)=+\infty$ if there are no weak solutions. On the other hand, if \eqref{o:acons} holds, and there exists a strong solution $v$, then the corresponding $\sup$ in \eqref{o:sadd1} is equal to \begin{multline*}\frac 12\int_0^T (\cf(v), \wei I)\,dt=\int_0^T \wei \left[K(t)+\frac 1 2 \int_\Omega \tr \cf(0)\, dx\right]\,dt\\=\int_0^T \wei \left[K(0)+\frac 1 2 \int_\Omega \tr \cf(0)\, dx\right]\,dt= \int_0^T \wei K_0\,dt=\Wei(0)K_0,\end{multline*} which yields $\mathcal I(v_0,T)\leq \Wei(0)K_0$. Finally, testing by $(E,B)=(0,0)$ we see that \begin{multline*}\tilde{\mathcal J}(v_0, T)\geq \frac 1 2 \inf_{v,M:\cf(v)\leq M}\int_0^T (M, \wei I)\,dt\geq \\ \inf_{v,M:\cf(v)\leq M}\int_{(0,T)\times\Omega} \wei \left(\ce(v)+\frac 1 2 \tr \cf(0) \right)\,dx\,dt\ge \frac {\Wei(0)|\Omega| \tr \cf(0)} 2  \geq 0.\end{multline*} \label{orweak}\end{remark} 

It is easy to see (taking into account that  $M+A\geq M$ for any $A\geql 0$) that if $\wei I+2B$ is not non-negative-definite on a set of positive Lebesgue measure in $(0,T)\times \Omega$, then the $\inf$ in \eqref{o:sadd2sub} equals $-\infty$. Hence,  any solution to \eqref{o:sadd2sub}  necessarily satisfies \be\label{o:bwe} \wei I+2B\geql 0\ \mathrm{a.e.}\ \mathrm{in}\ (0,T)\times \Omega.\ee
Consider the nonlinear functional $$\mathcal K: \lks((0,T)\times \Omega;\R^n)\times L^\infty((0,T)\times \Omega;\R^{\nn}_s)\to \R$$ defined by the formula
\be\label{o:defk1} \mathcal K(E,B)=\inf_{\cf(z)\le M}\int_0^T \left[(z,E)+\frac 1 2(M, \wei I+2B)\right]\,dt,\ee  where the infimum is taken over all pairs $(z,M)\in \lk((0,T)\times \Omega;\R^n)\times L^1((0,T)\times \Omega;\R^{\nn}_s)$. 


Then \eqref{o:sadd2sub} is equivalent to \be\label{o:conc}\tilde{\mathcal J}(v_0, T)=\sup_{E,B:\,\eqref{o:constrweak},\eqref{o:bwe}}\int_0^T -(v_0,E)\,dt+\bigk,\ee the supremum is taken over all pairs $(E,B)$ belonging to the class \eqref{o:be}.

\begin{remark}[A variant of our theory based on the relative entropy] \label{remach} Inspired by \cite{Ach7}, is possible to develop a variant of our theory for which the primal problem consists in minimizing a suitable time integral of the \emph{relative entropy} (also known as the Bregman divergence) \begin{equation} \label{o:consprem} \cb(t):=\int_\Omega \left[\ce (v(t,x))-\ce (\tilde v(t,x))-\tilde v^{\#}(t,x)\cdot (v(t,x)-\tilde v(t,x))\right]\,d x.\end{equation} \label{r:zar} Here $$\tilde v:[0,T]\times \Omega\to \R^n$$ is a fixed ``base state''. (Since $\ce(0)=0$ and $0^{\#}=0$, the setting that we have adopted in this paper corresponds to $\tilde v \equiv 0$.) The implementation of this idea lies beyond the scope of this paper, but it is very plausible that an existence theorem (similar to Theorem \ref{t:ex}) and a global in time consistency result similar to \cite[Theorem 3.6]{Ach7}, cf. Remark \ref{remcons}, should hold for the corresponing dual problem. An interesting task would be to obtain consistency results in which $\tilde v$ would neither be identically zero (as in Theorem \ref{o:smooth}) nor coincide with the strong solution (as in \cite[Theorem 3.6]{Ach7}).  \end{remark} 

\section{Consistency and Dafermos' principle} \label{Secx}

The following theorem shows that a strong solution determines a solution to the optimization problem \eqref{o:conc}, and vice versa.  This advocates the possibility to view the maximizers of \eqref{o:conc} as ``dual variational solutions`'' to \eqref{o:aeuler}, cf. \cite{CMP18,V22,Ach7}, see also Remark \ref{r:time}.

\begin{theorem}[Consistency] \label{o:smooth} Let $v_0\in \lk(\Omega;\R^n)$.  Let $v$ be a strong solution to \eqref{o:aeuler}  satisfying \be\label{o:pd++} \wei I  \geql  -2\Wei (t) L^*(\mom)\ \textrm{a.e.}\  \textrm{in}\ (0,T)\times \Omega.\ee  Then $\mathcal I(v_0,T)=\tilde{\mathcal J}(v_0,T)=\tilde{\mathcal I}(v_0,T)=\mathcal J(v_0,T)=\Wei(0)K_0$. The pair $(E_+,B_+)$ defined by \be B_+=L^*a,\, E_+=\p_t a, \label{e:beee}\ee where \be \label{axi} a=\Wei \mom,\,\ee belongs to the class \eqref{o:be} and maximizes \eqref{o:conc}.  Moreover, one can invert these formulas and express $v$ in terms of $E_+$ as follows
	 \be \label{o:ta1} v(t,x)=\nabla \ce ^*\left(\frac 1 {\Wei(t)} \int_{t}^T (-E_+)(s,x)\,ds\right),\quad t< T.\ee

\end{theorem}

\begin{proof} 
	By construction, the pair $(E_+,B_+)$ belongs to $\lks((0,T)\times \Omega;\R^n)\times L^\infty((0,T)\times \Omega;\R^{\nn}_s)$. It follows from $\Wei(T)=0$ that this pair verifies  \eqref{o:constrweak}. Moreover, \eqref{o:pd++} implies \eqref{o:bwe} for $B_+$. We now claim that \be \label{o:vret1} \frac 1 2 (\wei I+2B_+):\partial_l\cf(v)+(E_+)_l=0.\ee  Indeed, using \eqref{o:aeuler2} we compute \begin{multline}\frac 1 2 (\wei I+2B_+):\partial_l\cf(v)+(E_+)_l\\=\frac 1 2 (\wei I+2\Wei L^*(\mom)):\partial_l\cf(v)+(-\wei (\mom)_l+ \Wei \p_t (\mom)_l)\\ =\wei \partial_l\ce(v) +\Wei L^*(\mom):\partial_l\cf(v)+(-\wei (\mom)_l+ \Wei \p_t (\mom)_l) \\= \Wei( L^*(\mom:\partial_l\cf(v)+ \p_t (\mom)_l)
		=0.\end{multline}
	
	On the other hand, since $v$ is in particular a weak solution, it satisfies \eqref{o:w2} with test functions $(E_+,B_+)$. Thus we have \be \label{o:w2+}\int_0^T \left[(v-v_0,E_+)+(\cf(v), B_+)\right]\, dt=0.\ee Hence, by \eqref{o:vret1}, \be \label{f45} \int_0^T \left[-(v_0,E_+)+(\cf(v), B_+)\right]\, dt=\frac 1 2 \sum_{l=1}^n \int_0^T (v_l(\wei I+2B_+),\partial_l\cf(v))\, dt.\ee Employing \eqref{f45} and Lemma \ref{lemm}, we obtain \begin{multline} \label{o:w3+}\frac 1 2 \int_0^T (\cf(v),\wei I+2B_+)\, dt -\frac 1 2 \sum_{l=1}^n \int_0^T (v_l\partial_l\cf(v),\wei I+2B_+)\, dt \\= \int_0^T \left[(v_0,E_+)+\wei(t)\left(\ce (v(t))+\frac 1 2 \tr \cf(0), 1\right)\right]\, dt\\= \int_0^T \left[(v_0,E_+)+\wei(t)\left(\ce (v_0)+\frac 1 2 \tr \cf(0), 1\right)\right]\, dt\\=\int_0^T (v_0,E_+)\, dt+ \Wei(0)K_0.\end{multline}

	By Remark \ref{orweak}, we have $\mathcal I(v_0, T)\leq \Wei(0)K_0$. Hence, it suffices to show that \be \label{o:claim1} \int_0^T -(v_0,E_+)\, dt+\bigkp\ge \Wei(0)K_0, \ee so that there is no duality gap.
	
	The $\Lam$-convexity of $\cf$ implies that the function $y\mapsto \frac 1 2 \cf(y):(\wei I+2B_+(t,x))$ is convex w.r.t. $y\in \R^n$ for a.e. fixed pair of parameters $(t,x)\in (0,T)\times \Omega$.
	 Now \eqref{o:vret1}, \eqref{o:w3+} and \eqref{o:bwe} for $B_+$ yield \begin{multline*} \int_0^T -(v_0,E_+)\, dt+\bigkp\\=\int_0^T -(v_0,E_+)\, dt\\+\inf_{\cf(z)\le M}\int_0^T \left[-\frac 1 2\sum_{l=1}^n\left(z_l \partial_l\cf(v),\wei I+2B_+\right)+\frac 1 2(M, \wei I+2B_+)\right]\, dt\\ \ge \int_0^T -(v_0,E_+)\, dt\\+\inf_{z\in \lk((0,T)\times \Omega;\R^n)}\int_0^T \left[-\frac 1 2\sum_{l=1}^n\left(z_l \partial_l\cf(v),\wei I+2B_+\right)+\frac 1 2(\cf(z), \wei I+2B_+)\right]\, dt\\
		\\=\int_0^T -(v_0,E_+)\, dt\\+\int_0^T \left[-\frac 1 2\sum_{l=1}^n\left(v_l \partial_l\cf(v),\wei I+2B_+\right)+\frac 1 2(\cf(v), \wei I+2B_+)\right]\, dt\\=\Wei(0)K_0. \end{multline*} The penultimate equality follows from the well-known fact that, given a convex function $\Phi:\R^n \to \R$, the function $y\mapsto \Phi (y)-y\cdot \nabla \Phi (\xi)$ attains its minimum at $y=\xi$. 
	
	Finally, since $\Wei(T)=0$,  we deduce from \eqref{e:beee} and \eqref{axi} that \be\int_t^T E_+(s,x)\,ds=\int_t^T \p_s(\Wei(s)\mom(s,x))\, ds=-\Wei(t)\mom(t,x),\ee which yields \eqref{o:ta1}. 
	\end{proof}

\begin{remark}[Limitation \eqref{o:pd++} can be overcome by adapting the weight $\wei$] \label{r:timeei} At first glance, condition \eqref{o:pd++} indicates that the interval for the which the consistency holds can be smaller than the interval $[0,T)$ on which the strong solution\footnote{We say ``the strong solution" because strong solutions are unique, see Theorem \ref{l.uniques}.} exists. However, the weight $\wei$ can be selected to guarantee the consistency on any interval $[0,T_1]$, $T_1<T$. Indeed, by Remark \ref{r:equiv} with $\wei_1
	\equiv 1$,  for any strong solution on $[0,T)$ we have that $(T-t)L^*(\mom(t,x))$ is essentially bounded. Therefore, for any $T_1<T$ there exists $\gamma>0$ such that \be \label{e:weiwei} \gamma I \geql  -2 L^*(\mom)\ee a.e. on $(0,T_1)\times \Omega$. It is easy to see from Definition \ref{d:strongclass} that $v$ (more accurately, the restriction $v|_{[0,T_1]\times \Omega}$) is a strong solution on  $[0,T_1]$. Consider the weight $\wei(t):=\exp(-\gamma t)$ and let $\Wei(t):=\int_t^{T_1} \wei(s)\, ds\geq 0$ (note that we are now working on the interval $[0,T_1]$). Observe that $$\gamma\Wei\leq \wei.$$ Because of \eqref{e:weiwei}, this implies that \eqref{o:pd++} holds on $(0,T_1)\times \Omega$.  \end{remark}

As an application of Theorem \ref{o:smooth} we establish a variant of Dafermos' principle for \eqref{o:aeuler} affirming that a strong solution dissipates the total entropy not slower and not later than any other subsolution.

\begin{theorem}[Dafermos' principle] \label{t:daf} Let $v_0\in \lk(\Omega;\R^n)$ . Let $v$ be a strong solution to \eqref{o:aeuler} on the interval $[0,T]$ with total entropy $K(t)= K(0)=\int_\Omega \ce (v_0)\,d x$, and let $(u,M)$ be a subsolution on  $[0,T]$ with total entropy $\tilde K(t)=\frac 1 2 (M(t,\cdot)-\cf(0),I)$.  Then for any $0\leq t_0<t_1\leq T$ it cannot simultaneously be that $\tilde K(t)\le K(t)$ for a.a. $t\in (0,t_1)$ and $\tilde K(t)<K(t)$ for a.a. $t\in(t_0,t_1)$. In particular, it is impossible that $\tilde K(t)<K(t)$ for a.a. $t\in(0,\epsilon)$, $\epsilon>0$. 
\end{theorem}

\begin{proof}  We just prove the first claim, to obtain the second one it suffices to put $t_0=0$ and $t_1=\epsilon$.
	
	 Assume that a subsolution $(u,M)$ satisfies $\tilde K(t)\leq K(0)$ for a.a. $t\in(0,t_1)$ and $\tilde K(t)<K(0)$ for a.a. $t\in(t_0,t_1)$ for some $0\leq t_0<t_1\leq T$. W.l.o.g. $t_1<T$. Fix any $T_1\in (t_1,T)$. Let $\gamma$ be a sufficiently large number that satisfies \eqref{e:weiwei} a.e. on $(0,T_1)\times \Omega$ and other lower bounds to be determined below. As in Remark \ref{r:timeei}, we set $\wei(t):=\exp(-\gamma t)$ and $\Wei(t):=\int_t^{T_1} \wei(s)\, ds$, and infer that \eqref{o:pd++} holds on $(0,T_1)\times \Omega$.

 We now claim that \be \label{odgoal5} \int_{t_0}^{T_1} \wei(t) (\tilde K(t)-K(t))\, dt< 0,\ee which implies that \be \label{odgoal2} \int_0^{T_1} \wei(t) (\tilde K(t)-K(t))\, dt< 0.\ee  Since $v$ is a strong solution, its total entropy $K(t)=K(0)$, so \eqref{odgoal2} is equivalent to  $$ \int_0^{T_1} \wei(t) \tilde K(t)\, dt<\int_0^{T_1} \wei(t) K(0)\, dt=\Wei(0)K(0).$$ Hence, \be \label{odgoal}\frac 1 2 \int_0^{T_1} \wei(t)(M(t,\cdot),I)\, dt\\ <\frac 1 2 \int_0^{T_1}\wei(t) (\cf(0),I)\, dt+ \Wei(0)K(0)=\Wei(0)K_0.\ee
It follows from Theorem \ref{o:smooth} that \be \label{e:leve} \tilde{\mathcal I}(v_0,T_1)=\Wei(0)K_0.\ee It is easy to conclude from Definition \ref{d:ss} that the restriction $(u,M)|_{(0,T_1)\times \Omega}$ is a subsolution on  $(0,T_1)$. Thus in order to get a contradiction it is enough to put together \eqref{o:sadd1sub} with $T=T_1$, \eqref{odgoal} and \eqref{e:leve}.

	It remains to prove \eqref{odgoal5}. 
	Let $$\epsilon:=\int_{t_1}^{T_1} \wei(t) (\tilde K(t)-K(t))\, dt.$$
	Then our claim \eqref{odgoal5} becomes \be \label{odgoal3} \int_{t_0}^{t_1} \wei(t) (\tilde K(t)-K(t))\, dt< -\epsilon.\ee 
	If $\epsilon\leq 0$, \eqref{odgoal3} is obvious, so let us assume $\epsilon>0$. Observe now that $$\epsilon\leq \wei(t_1) \int_{0}^{T_1}  |\tilde K(t)-K(t)|\, dt\leq C_\epsilon \wei(t_1),$$ where $C_\epsilon$ does not depend on $\gamma$. Fix any point $t_2\in (t_0,t_1)$. Let $\gamma$ be so large that $$\exp(\gamma(t_1-t_2))\int_{t_0}^{t_2} ( K(t)-\tilde K(t))\, dt\ge C_\epsilon.$$  Then we conclude that \begin{multline*} \int_{t_0}^{t_1} \wei(t) (\tilde K(t)-K(t))\, dt<  \int_{t_0}^{t_2} \wei(t) (\tilde K(t)-K(t))\, dt\\ \leq \wei(t_2) \int_{t_0}^{t_2}  (\tilde K(t)-K(t))\, dt=\exp(-\gamma t_2) \int_{t_0}^{t_2}  (\tilde K(t)-K(t))\, dt \\ \leq -\exp(-\gamma t_1)C_\epsilon\leq -\epsilon.\end{multline*}
\end{proof}

\begin{remark}[Quadratic nonlinearity and comparison with \cite{Ach7}] \label{remcons} The results of this section are new even for the quadratic case $\cf(v)=v\otimes v$, and are consequently applicable to the incompressible Euler equation, cf. \cite{CMP18,V22} (strictly speaking, the abstract setting of \cite{V22} slightly differs from \eqref{o:aeuler} and reads $\p_t v=P L(v\otimes v)$, where $P$ is a suitable projector that in particular can be the Leray–Helmholtz projector; however, in the quadratic case the proofs of the current section work, mutatis mutandis, in the presence of the projector $P$). Another global in time consistency result for the dual formulation of the incompressible Euler has recently been obtained in \cite[Theorem 3.6]{Ach7}. Nevertheless, our results (Theorem \ref{o:smooth} and Remark \ref{r:timeei}) seem to be of a different nature, and, furthermore, the two attitudes complement each other to a certain degree. Indeed, \cite[Theorem 3.6]{Ach7} states that if the strong solution $v$ to the incompressible Euler  coincides with the ``base state''  $\tilde v$, then the solution of the dual problem is identically zero (in the variables used in \cite{Ach7}).  Thus the information about the strong solution is contained not in the solution of the dual problem but in the ``base state'' only.  Moreover, the proof of \cite[Theorem 3.6]{Ach7} ignores the formal conservativity of the problem (and therefore can be extended to the Navier-Stokes).   In contrast, our ``base state'' is identically zero (see Remark \ref{remach}), and the information about the strong solution is contained in the solution $(B_+,E_+)$ of the dual problem. This information can be retrieved by formula \eqref{o:ta1} that strongly relies on the formal conservativity  \eqref{o:acons}. \end{remark}

	
\section{Solvability of the dual problem} \label{Secsol}
		In this section we prove solvability of the dual problem \eqref{o:conc} under an additional technical assumption on the operator $L$. 
		\begin{definition}[Strong trace condition] \label{deftr} The operator $L$ is said to satisfy the strong trace condition if there exists a uniform constant $c$ such that for any $\zeta\in D(L^*)$ such that  the eigenvalues of the matrix $-L^* \zeta(x)$  are uniformly bounded from above by a constant $k$ for a.e. $x\in \Omega$, the eigenvalues of the matrix $L^* \zeta(x)$ are also uniformly bounded from above a.e. in $\Omega$ by $ck$.\end{definition}
		\begin{remark}\label{r:tracel} The strong trace condition is particularly satisfied provided \be \label{e:invt} L(qI)=0\ee for any smooth scalar function $q(x)$. Indeed, it suffices to observe that in this case the trace of the matrix function $L^*\zeta$ vanishes almost everywhere because $$(\tr(L^*\zeta),q)=(L^*\zeta,qI)=(\zeta,L(qI))=0.$$ \end{remark}
		\begin{remark} \label{r:tracel2} Although the strong trace condition is not very restrictive and holds in many situations, cf. Section \ref{Sec3}, see also \cite{V22}, it fails for the system of conservation laws \eqref{e:conslaw} with a symmetric flux matrix $\cf$. Indeed, the adjoint of $-\di$ is the symmetric part of the Jacobian matrix, and there is no way to control its eigenvalues from above if they are bounded from below.  However, in many cases it possible to recast the problem in a form that admits the strong trace condition. In Section \ref{Sec3} we will show how to implement this for the scalar conservation laws. 
			\end{remark}

		\begin{theorem}[Existence in anisotropic Orlicz spaces] \label{t:ex} Assume that $L$ satisfies the strong trace condition. Then for any $v_0\in \lk(\Omega;\R^n)$ there exists a maximizer $(E,B)$ to \eqref{o:conc} in the class \eqref{o:be}, and $$0\le\frac {\Wei(0)|\Omega| \tr \cf(0)} 2 \leq  \tilde{\mathcal J}(v_0, T)< +\infty.$$  \end{theorem}  
		\begin{proof} It suffices to consider the pairs $(E,B)$ that meet the restrictions \eqref{o:constrweak}, \eqref{o:bwe}. Testing \eqref{o:conc} with $E=0,\;B=0$ as in Remark \ref{orweak}, we see that $$ \tilde{\mathcal J}(v_0,T)\ge \frac {\Wei(0)|\Omega| \tr \cf(0)} 2 \ge 0.$$ Let $(E_m,B_m)$ be a maximizing sequence.  Since $0\le \tilde{\mathcal J}(v_0, T)$, without loss of generality we may assume that \be\label{o:ms}0 \le -\int_0^T (v_0,E_m)\, dt+\bigkm.\ee The eigenvalues of $-B_m$ are uniformly bounded from above because $\wei I+2B_m\geql 0$. Since the strong trace condition is assumed, a uniform $L^\infty$ bound on $B_m$ follows directly from Definition \ref{deftr}. In other words, $\wei I+2B_m\leql \wei c I$ with some constant $c>0$ a.e. in $(0,T)\times\Omega$. By the definition of $\mathcal K$ in \eqref{o:defk1}, we have \begin{multline}\label{e:42} \bigkm\leq  \inf_{\cf(z)\le M}\int_0^T\left[(z,E_m)+\frac {\wei c} 2(M, I)\right]\,dt \\ = \inf_{z\in \lk}\int_0^T\left[(z,E_m)+\left(\wei c, \ce (z)+\frac 1 2 \tr \cf(0)\right) \right]\,dt  \\ =-\int_0^T \left(\wei c,\ce^*\left(-\frac {E_m}{\wei c}\right)\right)\,dt+\frac {c\Wei(0)|\Omega| \tr \cf(0)} 2 .\end{multline} Employing the Fenchel–Young  inequality we infer that 
			\begin{multline} \label{e:43} \int_0^T \left(\wei c,\ce^*\left(-\frac {E_m}{\wei c}\right)\right)\,dt \leq -\int_0^T (v_0,E_m)\, dt+\frac {c\Wei(0)|\Omega| \tr \cf(0)} 2\\=\int_0^T \frac {\wei c}2 \left( 2v_0,-\frac {E_m}{\wei c}\right)\, dt +\frac {c\Wei(0)|\Omega| \tr \cf(0)} 2\\ \leq  \int_0^T \frac {1}2 (\ce(2v_0),\wei c)\,dt+ \int_0^T \frac {1}2 \left(\ce^*\left(-\frac {E_m}{\wei c}\right),\wei c\right)\,dt+\frac {c\Wei(0)|\Omega| \tr \cf(0)} 2,\end{multline} whence \be \label{o:ms2}  \int_0^T \frac {1}2 \left(\ce^*\left(-\frac {E_m}{\wei c}\right),\wei c\right)\,dt \leq  \Wei(0)C,\ee where $C$ depends only on the $L^1$-norm of $\ce(2 v_0)$ that is finite. This together with \eqref{e:modcont} yields a uniform $\lks((0,T)\times \Omega;\R^n)$-bound on $E_m$. It follows from \eqref{e:42}, \eqref{e:43} and \eqref{o:ms2} that the right-hand side of \eqref{o:ms} is uniformly bounded, whence $\tilde{\mathcal J}(v_0, T)<+\infty$. The functional $\mathcal K$ is concave and upper semicontinuous on $\lks((0,T)\times \Omega;\R^n)\times L^\infty((0,T)\times \Omega;\R^{n\times n}_s)$ as an infimum of affine continuous functionals, cf.  \eqref{o:defk1}. The functional $\int_0^T (v_0,\cdot)\, dt$ is a linear bounded functional on $\lks((0,T)\times \Omega;\R^n)$. Consequently, every weak-$*$ accumulation point of $(E_m,B_m)$ is a maximizer of \eqref{o:conc}. Note that the constraints \eqref{o:constrweak}, \eqref{o:bwe} are preserved by the limit. 
		\end{proof}

		\begin{remark}[Back to the original problem] \label{r:time} Let $(E,B)$ be any maximizer of \eqref{o:conc} satisfying \eqref{o:be}. Mimicking \eqref{o:ta1}, we can define a ``generalized solution'' to \eqref{o:aeuler} by setting \be\label{o:geuler} \mom(t,x):=-\frac 1 {\Wei(t)} \int_{t}^T E(s,x)\,ds,\\ v(t,x):=\nabla \ce^*(\mom(t,x)).\ee  This object $v$ automatically belongs to the same regularity class as the strong solutions. Indeed, $$\partial_t [\Wei \mom]=E\in L_{\ce^*}((0,T)\times \Omega; \R^n).$$ On the other hand, since $$\frac 1{\Wei(t)}\lesssim \frac 1 {T-t},$$ our anisotropic Hardy inequality (see Appendix \ref{la1}) implies that $\mom \in \lks((0,T)\times \Omega; \R^n)$, whence $v \in L_{\ce}((0,T)\times \Omega; \R^n)$ by \eqref{e:lham2}. Finally, $\mom(t)\in D(L^*)$ for a.a. $t\in (0,T)$, and \begin{equation} \label{e:bandv}\Wei L^* (\mom)=B\in L^\infty((0,T)\times \Omega; \R^{n\times n}_s).\end{equation} Indeed, let $\psi:(0,T)\times \Omega\to \R^{\nn}_s$ be an arbitrary smooth compactly supported matrix field, and set $\Psi(t,\cdot):=\int_0^t \psi(\tau,\cdot)\,d\tau$; integrating by parts  and using \eqref{o:constrweak}, we deduce that \begin{multline*} \int_0^T\left(B(t,\cdot),\psi(t,\cdot)\right)\,dt= -\int_0^T\left(E(t,\cdot),L\Psi(t,\cdot)\right)\,dt=
				\\=-\int_0^T\left(\int_t^T E(\tau,\cdot)\,d\tau,L\psi(t,\cdot)\right)\,dt=\int_0^T (\Wei (t)L^* \mom (t,\cdot),\psi(t,\cdot))\,dt.\end{multline*}
			However, at this level of generality, $v$ is not necessarily a strong solution. 
			\end{remark}

			\section{Uniqueness of strong solutions} \label{Secun} We believe that it is possible to prove the weak-strong uniqueness property in the sense that existence of a strong solution (as in Definition \ref{d:strongclass}) implies that the corresponding solution \eqref{e:beee} to the dual problem \eqref{o:sadd2sub} is unique in the class \eqref{o:be}. The corresponding result in the quadratic case $\cf(v)=v\otimes v$ was established in \cite[Section 5]{V22}. As the first step towards this conjecture, we show here that a strong solution is always unique. The proof heavily relies on $\Lam$-convexity of $\cf$.
			
			\begin{theorem}[Uniqueness of strong solutions]\label{l.uniques}  Let $u$, $v$ be two strong solutions to  \eqref{o:aeuler} with the same initial datum $v_0\in \lk(\Omega;\R^n)$. Then, for every $t\in [0,T)$, $u(t,\cdot)=v(t,\cdot)$ a.e. in $\Omega$.
			\end{theorem}
			\begin{proof} To fix the ideas, we will assume that $u$, $v$ are regular enough to perform the manipulations below. (The proof in the general case follows exactly the same strategy with some tedious and rather standard technicalities:  in the (comparatively easy) quadratic case $\cf(v)=v\otimes v$ the rigorous implementation can be found in \cite[Lemma 5.2]{V22}). In order to avoid heavy notation, we will write $u(t)$ and $v(t)$, and often even $u$ and $v$, instead of $u(t,\cdot)$ and $v(t,\cdot)$. 
				
				The key idea is to look at the evolution of the ``Jeffreys divergence''\footnote{The classical Jeffreys divergence corresponds to the case $n=1$ and $\ce(v)=v \log v$, which is however ruled out by our assumptions.} $$\je(t):=(u(t)-v(t),\mou(t)- \mom(t))=(u(t)-v(t),\nabla \ce (u)(t)- \nabla \ce (v)(t)).$$ By strict convexity of $\ce$, $\je(t)\geq 0$, and $\je(t)=0$ if and only if $u(t)=v(t)$. Obviously, $\je(0)=0$, and it is enough to show that \be \label{e51}\je(t)= 0, \ t\in [0,T_1]\ee for any $T_1<T$. 
				
				Using \eqref{o:aeuler} and \eqref{o:aeuler2}, we compute \begin{multline}\je'(t)=(\p_t u(t)-\p_t v(t),\mou(t)- \mom(t))+(u(t)-v(t),\p_t\mou(t)- \p_t\mom(t))\\=(\cf (u)-\cf (v),L^* (\mou)- L^*(\mom))\\+\sum_{l=1}^n(u_l-v_l,L^* (\mom):\partial_l \cf(v)- L^* (\mou):\partial_l \cf(u))\\=\left(-L^*(\mou),\cf (v)-\cf(u)-\sum_{l=1}^n (v_l-u_l):\partial_l \cf(u)\right)\\+\left(-L^*(\mom),\cf (u)-\cf(v)-\sum_{l=1}^n (u_l-v_l):\partial_l \cf(v)\right).\end{multline}  It follows from Definition \ref{d:strongclass}, cf. \eqref{o:strongclass}, that $L^*(\mou)$ and $L^*(\mom)$ are essentially bounded on $[0,T_1]\times \Omega$. Thus there exists a constant $\gamma>0$ such that $\gamma I-L^*(\mou)$ and $\gamma I-L^*(\mom)$ are positive-definite matrix functions a.e. on $[0,T_1]\times \Omega$, and for some uniform $c>0$ we have $-L^*(\mou)-L^*(\mom)\leql cI$.
				Consequently, \begin{multline}\label{e52}\je'(t)=\left(\gamma I-L^*(\mou),\cf (v)-\cf(u)-\sum_{l=1}^n (v_l-u_l):\partial_l \cf(u)\right)\\+\left(\gamma I-L^*(\mom),\cf (u)-\cf(v)-\sum_{l=1}^n (u_l-v_l):\partial_l \cf(v)\right)\\-\gamma \sum_{l=1}^n \left(u_l-v_l,\tr[\partial_l \cf(u)-\partial_l \cf(v)]\right).\end{multline}
				
				Due to the $\Lam$-convexity of $\cf$, the functions $y\mapsto (\gamma I-L^*(\mou(t,x))):\cf(y)$ and $y\mapsto (\gamma I-L^*(\mom(t,x))):\cf(y)$ are convex w.r.t. $y\in \R^n$ for a.e. fixed pair of parameters $(t,x)\in [0,T_1]\times \Omega$. On the other hand, it is easy to see that for any convex function $\Phi:\R^n \to \R$ we have \be \label{e53} 0\leq \Phi(y)-\Phi (\xi)-(y-\xi)\cdot \nabla \Phi (\xi)\leq (y-\xi)\cdot (\nabla \Phi (y)-\nabla \Phi (\xi)), \ y,\xi\in\R^n.\ee Leveraging \eqref{e53}, we derive from \eqref{e52} that \begin{multline*}\je'(t)\le \left(\gamma I-L^*(\mou),\sum_{l=1}^n (v_l-u_l):(\partial_l \cf(v)-\partial_l \cf(u))\right)\\+\left(\gamma I-L^*(\mom),\sum_{l=1}^n (u_l-v_l):(\partial_l \cf(u)-\partial_l \cf(v))\right)\\-2\gamma \sum_{l=1}^n \left(u_l-v_l,\partial_l \ce(u)-\partial_l \ce(v)\right)\\ \leq \left((2\gamma+c) I,\sum_{l=1}^n (u_l-v_l):(\partial_l \cf(u)-\partial_l \cf(v))\right)-2\gamma  \left(u-v,\mou-\mom \right)\\=(4\gamma+2c) \sum_{l=1}^n \left(u_l-v_l,\partial_l \ce(u)-\partial_l \ce(v)\right)-2\gamma \je(t)=2(c+\gamma)\je(t),\end{multline*} and \eqref{e51} follows by Gr\"onwall's inequality. \end{proof}

\section{Applications to PDEs} \label{Sec3}

In this section present three relatively simple examples of dispersive equations to which our abstract theory is germane (NLS, NLKG and GKdV equations with generic defocusing power-law nonlinearities). We however believe that it is possible to go much beyond these examples by considering PDEs of higher order (as, e.g., in \cite{KS97,BL17}), with nonlocal terms (e.g., of Benjamin-Ono or Hartree type) or more general nonlinearities. In all these examples, we will be able to verify the assumptions of Section \ref{s:abst} and of Theorem \ref{t:ex}. Consequently, the theorems of Sections \ref{Secx}, \ref{Secsol} and \ref{Secun} are fully applicable here. 

To the best of our knowledge, global in time solvability for these PDEs is merely known under significant restrictions on the exponents in the power laws or on the size of the initial data, cf. \cite{sulem2007nonlinear,LP14,FRR23}. Theorem \ref{t:ex} and Remark \ref{r:time}  provide existence of certain ``dual variational solutions" for these problems without such restrictions. 

As a warm-up exercise, we show how our theory can be applied to scalar conservation laws. 

\subsection{Scalar conservation laws} In this example we further restrict ourselves to $\Omega=\mathbb T^1$. Consider the scalar conservation law
\begin{equation}\label{conslawnew} \p_t v= -\p_x (\ce(v)),\end{equation} where $\ce:\R\to\R$ is as in Assumption \ref{ass1}. Note that there is a tiny ``anisotropy'' because $\ce$ does not need to be even. As anticipated in Remark \ref{r:tracel2}, the operator $-\p_x$ does not satisfy the strong trace condition, so Theorem \ref{t:ex} is not directly applicable. In order to overcome this obstacle, we let $$n=1, N=2,$$
$$\cf(v)=\left(\begin{array}{@{}c|c@{}}
	  \ce(v) & \ce(v)  \\ \hline \ce(v) & \ce(v)
\end{array}\right).$$
Obviously, $\ce=\frac 1 2 \tr (\cf-\cf(0))$ and $\mom=\ce'(v)$.

We now define the operator $L$ by the formula $$ L\,\left(\begin{array}{@{}c|c@{}}
	a_{11} & a_{12}  \\  \hline 
	a_{12} & a_{22}  
\end{array}\right)=-\p_x a_{12}.$$ Then we can rewrite \eqref{conslawnew} in the abstract form \eqref{o:aeuler} with $L$ and $\cf$ that we have introduced.

Condition \eqref{e:invt} is obviously true, so the strong trace condition holds. Moreover, $$L^*(\kappa)=\frac 1 2 \left(\begin{array}{@{}c|c@{}}
	0 & \p_x \kappa  \\  \hline 
	\p_x \kappa  & 0 
\end{array}\right).$$

Let us check  the conservativity condition \eqref{o:acons}.  Remembering that we are now working on $\mathbb T^1$, we easily see that  $$ (\cf(v), L^*(\mom))=(\ce(v),\p_x (\ce'(v)))=-(\p_x v, (\ce'(v))^2)=0.$$

It is clear that $\cf\geql 0$ and $\cf$ is Loewner convex. The subspace $\Lam$ consists of the elements of the form
$$\left(\begin{array}{@{}c|c@{}}
	a_{11} & a_{12}  \\  \hline 
	a_{12} & a_{11}
\end{array}\right).$$ 
The validity of Assumptions  \ref{convf} and \ref{assl} is now obvious.

\subsection{GKdV}
We keep assuming $\Omega=\mathbb T^1$. Consider the  ``defocusing'' generalized Korteweg-de Vries equation  \cite{S81,FRR23,MS18} (GKdV)\footnote{To simplify the presentation we consider the real-valued equation, but the same approach works for the complex-valued case.}  \be\label{e:kdvxi} \p_t u+\p_{\xi\xi\xi} u=|u|^\alpha\p_\xi u, \quad u(0)=u_0.\ee The unknown is $u:[0,T]\times \Omega\to \R$, and $\alpha\geq 1$ is a prescribed constant (not necessarily integer).  Performing a simple change of variable $x:=\xi-t$ we rewrite this equation as follows: \be\label{e:kdv} \p_t u+\p_{xxx} u=(|u|^\alpha+1)\p_x u, \quad u(0)=u_0.\ee 

Letting $w:=\p_x u$ we can rewrite the problem in the following way: \begin{gather}\label{e:kdv1} \p_t u=-\p_{xx} w+\frac 1 {\alpha +1} \p_x(|u|^\alpha u)+\p_x u,\\\label{e:kdv2} \p_t w=-\p_{xxx} w+\frac 1 {\alpha +1} \p_{xx}(|u|^\alpha u)+\p_{xx} u,\\ \label{e:kdv3} \quad u(0)=u_0,\ w(0)=w_0.\end{gather}
Rigorously speaking, this is an extended system, but what is important is that it still possesses an ``entropy'' that is formally conserved even if we do not assume the compatibility condition $w_0=\p_x u_0$.

We let $$n=2, N=4,$$ $$ v=(u,w),$$ $$\bar v=(1,u,w,|u|^{\alpha}),$$ $$\ce(v)=\frac 1 2 (u^2+w^2)+\frac {|u|^{\alpha+2}}{(\alpha+2)(\alpha+1)},$$ \begin{multline*}\cf(v)=\epsilon \bar v\otimes \bar v+\left(\frac 12 \ce(v) + 1\right)e_1\otimes e_1+ \left(-\epsilon u^2+\frac 12 \ce(v)\right)  e_2\otimes e_2\\+\left(-\epsilon w^2+\frac 12 \ce(v)\right) e_3\otimes e_3+\left(-\epsilon|u|^{2\alpha}+\frac 12 \ce(v)\right)  e_4\otimes e_4,\end{multline*} where small $\epsilon>0$ will be selected later. It is straightforward to check that  $\ce=\frac 1 2 \tr (\cf-\cf(0))$ and that $\ce$ fits into the framework of Assumption \ref{ass1}. Furthermore, $$\mom=(z,w)=(u+(\alpha +1)^{-1} |u|^\alpha u,w);$$ observe that the second component coincides with $w$.
 
Define the operator $L$ by the formula $$ L\,\left(\begin{array}{@{}c|c|c|c@{}}
 	a_{11} & a_{12}  &  a_{13} & a_{14}  \\  \hline 
 	a_{12} & a_{22}  &  a_{23} & a_{24}   \\ \hline a_{13} & a_{23}  &  a_{33} & a_{34}   \\ \hline a_{14} & a_{24}  &  a_{34} & a_{44} 
 \end{array}\right)=\epsilon^{-1}\left(\begin{array}{@{}c@{}}
 	-\p_{xx} a_{13} +\frac 1 {\alpha +1} \p_{x}a_{24}+\p_x a_{12} \\ \hline 
 	-\p_{xxx} a_{13} +\frac 1 {\alpha +1} \p_{xx}a_{24}+\p_{xx} a_{12}
 \end{array}\right).$$ Then \eqref{e:kdv1}--\eqref{e:kdv2} can be written in the abstract form \eqref{o:aeuler}. 
 
 Condition \eqref{e:invt} is obviously true, so the strong trace condition holds. Moreover,
 
\begin{multline*} L^*\,\left(\begin{array}{@{}c@{}}
		\kappa \\ \hline 
		\theta 
	\end{array}\right)=\frac {\epsilon^{-1}}2 \\ \times \left(\begin{array}{@{}c|c|c|c@{}}
		0 & -\p_{x}\kappa + \p_{xx}\theta&-\p_{xx}\kappa +  \p_{xxx}\theta & 0   \\  \hline 
		-\p_{x}\kappa + \p_{xx}\theta & 0  &  0 & -\frac 1 {\alpha +1} \p_{x}\kappa  +\frac 1 {\alpha +1} \p_{xx}\theta   \\ \hline -\p_{xx}\kappa +  \p_{xxx}\theta & 0  &  0 & 0   \\ \hline 0 & -\frac 1 {\alpha +1} \p_{x}\kappa  +\frac 1 {\alpha +1} \p_{xx}\theta  &  0 & 0 
	\end{array}\right).\end{multline*}

Let us check  the conservativity condition \eqref{o:acons}. We compute, integrat\-ing by parts where needed, \begin{multline*}  (\cf(v), L^*(\mom))=\\ (u,-\p_{x}z + \p_{xx}w) +(w,-\p_{xx}z +  \p_{xxx}w)+\left(|u|^\alpha u,-\frac 1 {\alpha +1} \p_{x}z  +\frac 1 {\alpha +1} \p_{xx}w\right)\\=(u,-\p_{x}z + \p_{xx}w) +(\p_{xx}w,-z +  \p_{x}w)+\left(z-u,- \p_{x}z  + \p_{xx}w\right)\\=(\p_{xx}w, \p_{x}w)+\left(z,- \p_{x}z \right)=0.\end{multline*}
 
Observe that $\Lam$ consists of the elements of the form
 $$\left(\begin{array}{@{}c|c|c|c@{}}
 	a_{11} & a_{12}  &  a_{13} & 0\\  \hline 
 	a_{12} & a_{11}  &  0 & a_{24}   \\ \hline a_{13} & 0  &  a_{11} & 0  \\ \hline 0 & a_{24}  &  0 & a_{11}  
 \end{array}\right)$$
 Let us check that $\cf\geq 0$ and that $\cf$ is $\Lam$-convex. Fix any matrix $$ P=\left(\begin{array}{@{}c|c|c|c@{}}
 	p_{11} & p_{12}  &  p_{13} & 0\\  \hline 
 	p_{12} & p_{11}  &  0 & p_{24}   \\ \hline p_{13} & 0  &  p_{11} & 0  \\ \hline 0 & p_{24}  &  0 & p_{11}  
 \end{array}\right)\geql 0.$$ Then \begin{multline} \label{e:76}\cf(v):P=\epsilon p_{11}+\epsilon u^2 p_{11}+ \epsilon w^2 p_{11}+\epsilon|u|^{2\alpha} p_{11}\\+2 \epsilon u p_{12}+2 \epsilon w p_{13} +2 \epsilon|u|^{\alpha} u  p_{24}\\+\frac 1 2 \ce(v) p_{11}+p_{11}+ \left(-\epsilon u^2+\frac 1 2 \ce(v)\right)  p_{11}\\+\left(-\epsilon w^2+\frac 1 2 \ce(v)\right) p_{11}+\left(-\epsilon|u|^{2\alpha}+\frac 1 2 \ce(v)\right)  p_{11}\\= \left[\epsilon+1+u^2+w^2+2\frac {|u|^{\alpha+2}}{(\alpha+2)(\alpha+1)}\right]p_{11}+2 \epsilon u p_{12}+2 \epsilon w p_{13} +2 \epsilon|u|^{\alpha} u  p_{24}\end{multline}  Since $P \geql 0$, \be \max\{|p_{12}|, |p_{13}|, |p_{24}|\}\leq p_{11}.\ee Thus the function in \eqref{e:76} is non-negative provided $\epsilon$ is sufficiently small. The Hessian of the function in \eqref{e:76} w.r.t. $v=(u,w)$ is
 $$\left(\begin{array}{@{}c|c@{}}
  2 p_{11}+2 |u|^{\alpha}p_{11}+2\epsilon\alpha(\alpha+1) |u|^{\alpha-2} u p_{24} & 0
 	\\ \hline 0 & 2 p_{11}  
 \end{array}\right).$$ This matrix is always non-negative-definite provided $\epsilon$ is less than or equal to a certain constant that depends only on $\alpha$. 
 
We now we need to compute the partial derivatives of the non-diagonal components of $\lamproj\cf$ w.r.t. the components of $v$:
 $$\p_u \cf_{12}=\epsilon,$$
 $$\p_u \cf_{13}=0,$$
 $$\p_u \cf_{24}=\epsilon(\alpha+1)|u|^\alpha\lesssim |u|^{\alpha+1}+1\lesssim |\mom|+1,$$
 $$\p_w \cf_{12}=0,$$
 $$\p_w \cf_{13}=\epsilon,$$
 $$\p_w \cf_{24}=0.$$
 This immediately implies \eqref{e:lessss}.
 
 \begin{remark}[Lack of  Loewner convexity] \label{lowtoostrong} It is interesting to observe the full Loewner convexity is missing for the matrix function $\cf$  above. Indeed, let $$ P=\left(\begin{array}{@{}c|c|c|c@{}}
 		0& 0  &  0 & 0\\  \hline 
 		0 &0  &  0 & 0   \\ \hline 0 & 0  &  1 & 1  \\ \hline 0 & 0  &  1 & 1
 	\end{array}\right)\geql 0.$$ Then, for any $\alpha\geq 1$ and any fixed $\epsilon>0$, the function $$(u,w)\mapsto \cf(v):P=\ce(v)+2\epsilon|u|^{\alpha} w = \frac 1 2 (u^2+w^2)+\frac {|u|^{\alpha+2}}{(\alpha+2)(\alpha+1)}+2\epsilon|u|^{\alpha} w$$ is obviously not convex on $\R^2$. Similar counterexamples can be constructed for the systems discussed in the sequel. \end{remark}

\begin{remark}[Sharp formulation of GKdV] For any $\alpha\geq 1$, using the obtained expressions of the partial derivatives of $\cf$, it is elementary to see that the ``sharp'' problem \eqref{o:aeuler2} for GKdV reads \begin{gather*}\p_t z+(1+\Xi(|z|))(-\p_{x} z + \p_{xx} w)=0,\\ \p_t w-\p_{xx}z +  \p_{xxx}w=0,\end{gather*}  where $\Xi$ is the inverse of the function $s \mapsto s^{1/\alpha} (1+(\alpha+1)^{-1} s)$.\end{remark}

\subsection{Defocusing NLS} Let $\Omega=\mathbb T^d$. Consider the  nonlinear Schr\"odinger equation with a generic defocusing power-law nonlinearity
\be \iu \p_t\Psi=-\Delta \Psi +  |\Psi|^{2q}\Psi,\ \Psi(0)=\Psi_0,\ee where $q\geq \frac 1 2 $ is a given constant.   The unknown is $\Psi:[0,T]\times \Omega\to \mathbb C$. 

We first change the variable $\psi := \Psi e^{-\iu t}$  to rewrite this in the form \be \iu \p_t\psi=-\Delta \psi +  |\psi|^{2q}\psi+ \psi,\  \psi(0)=\Psi_0.\ee We now let $a=\Rea \psi, b=\Ima \psi, \aaa=\nabla a, \bbb=\nabla b$. Then the system becomes  \begin{gather}\label{nls1} \p_t a=-\di \bbb+(a^2+b^2)^q b+b,\\  \label{nls2}  \p_t b=\di \aaa -(a^2+b^2)^q a- a,\\ \label{nls3} \p_t \aaa=-\Delta \bbb+ \nabla ((a^2+b^2)^q b+ b),\\  \label{nls4} \p_t \bbb=\Delta \aaa- \nabla ((a^2+b^2)^q a- a). \end{gather} As in the previous example, this is an extended system, but we will see that it still possesses an ``entropy'' that is formally conserved.

We let $$n=2d+2, N=2d+4,$$ $$ v=(a,b,\aaa,\bbb),$$ $$\bar v=(1,a,b,\aaa,\bbb,(a^2+b^2)^q),$$
 $$\ce(v)=\frac 1 2 \left(|\aaa|^2+|\bbb|^2+ a^2+b^2+\frac 1 {q+1} (a^2+b^2)^{q+1}\right),$$ 
\begin{multline*}\cf(v)=\epsilon \bar v\otimes \bar v+\left (\frac 2 N \ce(v)+1\right)e_1\otimes e_1+ \left(-\epsilon a^2+\frac 2 N \ce(v)\right)  e_2\otimes e_2\\+\left(-\epsilon b^2+\frac 2 N \ce(v)\right) e_3\otimes e_3+\sum_{m=1}^d \left(-\epsilon|\aaa_m|^2+\frac 2 N \ce(v)\right)  e_{3+m}\otimes e_{3+m}\\+\sum_{m=1}^d \left(-\epsilon|\bbb_m|^2+\frac 2 N \ce(v)\right)  e_{3+d+m}\otimes e_{3+d+m}+\left(-\epsilon(a^2+b^2)^{2q}+\frac 2 N \ce(v)\right)  e_N\otimes e_N,\end{multline*} where small $\epsilon>0$ will be selected later. It is straightforward to check that  $\ce=\frac 1 2 \tr (\cf-\cf(0))$ and that 
$$\mom=(y,z,\aaa,\bbb)=(a+(a^2+b^2)^q a, b +(a^2+b^2)^q b,\aaa,\bbb).$$ Define the operator $L$ by the formula
$$ L\,\left(\begin{array}{@{}c|c|c|c|c|c@{}}
	a_{11} & a_{12}  &  a_{13} & A_{14} & A_{15} & a_{16} \\  \hline 
	a_{12} & a_{22}  &  a_{23} & A_{24} & A_{25} & a_{26} \\ \hline a_{13} & a_{23}  &  a_{33} & A_{34} & A_{35} & a_{36} \\ \hline A_{14}^\top & A_{24}^\top  &  A_{34}^\top & A_{44} & A_{45} & A_{46} 
	\\ \hline A_{15}^\top & A_{25}^\top  &  A_{35}^\top & A_{45}^\top & A_{55} & A_{56} \\ \hline a_{16} & a_{26}  &  a_{36} & A_{46}^\top & A_{56}^\top & a_{66}
\end{array}\right)=\epsilon^{-1}\left(\begin{array}{@{}c@{}}
	-\di A_{15} +a_{36}+ a_{13}\\ \hline 
	\di A_{14} -a_{26}- a_{12}\\ \hline -\Delta A_{15} +\nabla  (a_{36}+ a_{13})\\ \hline \Delta A_{14}-\nabla (a_{26}+a_{12})
\end{array}\right).$$
Then \eqref{nls1}--\eqref{nls4} can be rewritten in the abstract form \eqref{o:aeuler}. 

 Condition \eqref{e:invt} is obviously true, so the strong trace condition holds. Moreover,
\begin{multline*} L^*\,\left(\begin{array}{@{}c@{}}
	\kappa \\ \hline 
	\theta \\ \hline \eta \\ \hline \xi
\end{array}\right)=\frac {\epsilon^{-1}}2\\ \times \left(\begin{array}{@{}c|c|c|c|c|c@{}}
	0 & -\theta+\di \xi &  \kappa-\di \eta & -\nabla \theta+\Delta \xi & \nabla \kappa-\Delta \eta & 0 \\  \hline 
	-\theta+\di \xi & 0  &  0 & 0 & 0 & -\theta+\di \xi \\ \hline \kappa-\di \eta & 0  &  0 & 0 & 0 & \kappa-\di \eta \\ \hline (-\nabla \theta+\Delta \xi)^\top & 0  &  0 & 0 & 0 & 0
	\\ \hline (\nabla \kappa-\Delta \eta )^\top & 0  &  0 & 0 & 0 & 0 \\ \hline 0 & -\theta+\di \xi &  \kappa-\di \eta & 0 & 0 & 0
\end{array}\right).\end{multline*}
Let us check the conservativity condition \eqref{o:acons}. We compute, integrating by parts where needed,
\begin{multline*}  (\cf(v), L^*(\mom))=(a,-z+\di \bbb)+(b,y-\di \aaa)+(\aaa,-\nabla z+\Delta \bbb)+(\bbb,\nabla y-\Delta \aaa )\\ +(a(a^2+b^2)^{q},-z+\di \bbb)+(b(a^2+b^2)^{q},y-\di \aaa)\\=(y,-z+\di \bbb)+(z,y-\di \aaa) + (\di\aaa, z)-(\nabla \aaa,\nabla \bbb)- (\di\bbb, y)+(\nabla \bbb,\nabla \aaa)=0.\end{multline*}

Observe now that $\Lam$ consists of the elements of the form
$$\left(\begin{array}{@{}c|c|c|c|c|c@{}}
	a_{11} & a_{12}  &  a_{13} & A_{14} & A_{15} & 0 \\  \hline 
	a_{12} & a_{11}  &  0 & 0 & 0 & a_{26} \\ \hline a_{13} & 0  &  a_{11} & 0 & 0 & a_{36} \\ \hline A_{14}^\top & 0^\top  &  0^\top & a_{11} I & 0  & 0
	\\ \hline A_{15}^\top & 0  &  0 & 0 & a_{11} I  & 0 \\ \hline 0 & a_{26}  &  a_{36} & 0 & 0 & a_{11}
\end{array}\right)$$ Let us show that $\cf\geq 0$ and that $\cf$ is $\Lam$-convex.
Fix  $$P=\left(\begin{array}{@{}c|c|c|c|c|c@{}}
	p_{11} & p_{12}  &  p_{13} & P_{14} & P_{15} & 0 \\  \hline 
	p_{12} & p_{11}  &  0 & 0 & 0 & p_{26} \\ \hline p_{13} & 0  &  p_{11} & 0 & 0 & p_{36} \\ \hline P_{14}^\top & 0^\top  &  0^\top & p_{11} I & 0  & 0
	\\ \hline P_{15}^\top & 0  &  0 & 0 & p_{11} I  & 0 \\ \hline 0 & p_{26}  &  p_{36} & 0 & 0 & p_{11}
\end{array}\right)\geql 0.$$
Then, in the same spirit as for the GKdV, we compute \begin{multline} \label{e:767}\cf(v):P= \left[\epsilon+1+ \aaa^2+\bbb^2+ a^2+b^2+\frac 1 {q+1} (a^2+b^2)^{q+1} \right]p_{11}\\+2 \epsilon a p_{12}+2 \epsilon b p_{13} +2 \epsilon P_{14}\aaa+2 \epsilon P_{15}\bbb\\+2 \epsilon(a^2+b^2)^{q} a  p_{26}+2 \epsilon(a^2+b^2)^{q} b  p_{36}.\end{multline} Since $P\geql 0$, it is tedious but elementary to check that the function in \eqref{e:767} and its Hessian w.r.t. $v=(a,b,\aaa,\bbb)$ are non-negative and non-negative-definite, resp., provided $\epsilon$ is controlled by a constant that depends only on $q\ge \frac 1 2$. 

It remains to estimate the partial derivatives of the non-diagonal components of $\lamproj\cf$ w.r.t. the components of $v$. The only potentially problematic terms (because of the presence of a nonlinearity) are 
$$\p_a \cf_{26}=\epsilon(a^2+b^2)^{q}+2\epsilon qa^2(a^2+b^2)^{q-1},$$
$$\p_b \cf_{26}=2\epsilon qab(a^2+b^2)^{q-1},$$
$$\p_a \cf_{36}=2\epsilon qab(a^2+b^2)^{q-1},$$
$$\p_b \cf_{36}=\epsilon(a^2+b^2)^{q}+2\epsilon qb^2(a^2+b^2)^{q-1}.$$
All of them are $$\lesssim (a^2+b^2)^{q}\lesssim  (a^2+b^2)^{\frac {2q+1}{2}} +1 \lesssim |\mom|+1,$$  which yields validity of assumption \eqref{e:lessss}.

\section{Complex NLKG}  Let $\Omega=\mathbb T^d$. Consider the complex Klein-Gordon equation with a generic (defocusing) power-law nonlinearity  \be \p_{tt}\psi-\Delta \psi +  |\psi|^{2q}\psi+\psi=0,\ \psi(0)=\psi_0.\ee Here $q\geq \frac 1 2 $ is a given constant. The unknown is $\psi:[0,T]\times \Omega\to \mathbb C$. 

We let $a=\Rea \psi, b=\Ima \psi, \aaa=\nabla a, \bbb=\nabla b, u=\p_t a, v=\p_t b$. Then the system can be rewritten in the following form: \begin{gather} \p_t a=u,\\
	\p_t b=v,\\ \p_t \aaa =\nabla u,\\   \p_t \bbb=\nabla v,\\ 
	\p_t u=\di \aaa-(a^2+b^2)^q a- a,\\   \p_t v=\di \bbb -(a^2+b^2)^q b- b. \end{gather} As above, this is an extended system, but what matters is that it is formally conservative. 
	
We let $$n=2d+4, N=2d+6,\ v=(a,b,\aaa,\bbb,u,v),$$ $$\bar v=(1,a,b,\aaa,\bbb,u,v,(a^2+b^2)^q),$$
$$\ce(v)=\frac 1 2 \left(u^2+v^2+|\aaa|^2+|\bbb|^2+ a^2+b^2+\frac 1 {q+1}(a^2+b^2)^{q+1}\right),$$ 
\begin{multline*}\cf(v)=\epsilon \bar v\otimes \bar v+\left(\frac 2 N \ce(v) +1\right)e_1\otimes e_1+ \left(-\epsilon a^2+\frac 2 N \ce(v)\right)  e_2\otimes e_2\\+\left(-\epsilon b^2+\frac 2 N \ce(v)\right) e_3\otimes e_3+\sum_{m=1}^d \left(-\epsilon|\aaa_m|^2+\frac 2 N \ce(v)\right)  e_{3+m}\otimes e_{3+m}\\+\sum_{m=1}^d \left(-\epsilon|\bbb_m|^2+\frac 2 N \ce(v)\right)  e_{3+d+m}\otimes e_{3+d+m}\\+\left(-\epsilon u^2 +\frac 2 {N} \ce(v)\right)  e_{N-2}\otimes e_{N-2}\\+\left(-\epsilon v^2+\frac 2 {N} \ce(v)\right)  e_{N-1}\otimes e_{N-1}\\+\left(-\epsilon (a^2+b^2)^{2q}+\frac 2 {N} \ce(v)\right)  e_{N}\otimes e_{N},\end{multline*} where small $\epsilon>0$ will be selected later. It is straightforward to check that  $\ce=\frac 1 2 \tr (\cf-\cf(0))$.
We have
$$\mom=(y,z,\aaa,\bbb,u,v)=( a+a (a^2+b^2)^q , b +b (a^2+b^2)^q,\aaa,\bbb,u,v).$$ We now set
$$ L\,\left(\begin{array}{@{}c|c|c|c|c|c|c|c@{}}
	a_{11} & a_{12}  &  a_{13} & A_{14} & A_{15} & a_{16} & a_{17}& a_{18}\\  \hline 
	a_{12} & a_{22}  &  a_{23} & A_{24} & A_{25} & a_{26} & a_{27}& a_{28}\\ \hline a_{13} & a_{23}  &  a_{33} & A_{34} & A_{35} & a_{36}& a_{37}& a_{38} \\ \hline A_{14}^\top & A_{24}^\top  &  A_{34}^\top & A_{44} & A_{45} & A_{46} & A_{47}& A_{48}
	\\ \hline A_{15}^\top & A_{25}^\top  &  A_{35}^\top & A_{45}^\top & A_{55} & A_{56}& A_{57}& A_{58} \\ \hline a_{16} & a_{26}  &  a_{36} & A_{46}^\top & A_{56}^\top & a_{66} & a_{67}& a_{68} \\ \hline 
	a_{17} & a_{27}  &  a_{37} & A_{47} ^\top & A_{57}^\top & a_{67} & a_{77}& a_{78} \\ \hline
	a_{18} & a_{28}  &  a_{38} & A_{48}^\top & A_{58}^\top & a_{68} & a_{78}& a_{88}
\end{array}\right)= \epsilon^{-1} \left(\begin{array}{@{}c@{}}
	 a_{16}\\ \hline 
	a_{17}\\ \hline \nabla a_{16}\\ \hline \nabla a_{17} \\ \hline \di A_{14}-(a_{28}+a_{12})\\ \hline \di A_{15}-(a_{38}+a_{13})
\end{array}\right).$$
Then \eqref{nls1}--\eqref{nls4} can be rewritten in the abstract form \eqref{o:aeuler}. 

It is clear that
$$ L^*\,\left(\begin{array}{@{}c@{}}
	\kappa \\ \hline 
	\theta \\ \hline \eta \\ \hline \xi \\ \hline \chi \\ \hline \omega
\end{array}\right)=\frac {\epsilon^{-1}} 2 \left(\begin{array}{@{}c|c|c|c|c|c|c|c@{}}
	0 & -\chi &  -\omega& -\nabla \chi & -\nabla \omega & \kappa-\di \eta & \theta-\di \xi & 0 \\  \hline 
	-\chi & 0  &  0 & 0 & 0 & 0& 0& -\chi \\ \hline -\omega & 0  &  0 & 0 & 0 & 0& 0 & -\omega\\ \hline -\nabla \chi^\top & 0  &  0 & 0 & 0 & 0 & 0 & 0
	\\ \hline -\nabla \omega^\top & 0  &  0 & 0 & 0 & 0 & 0 & 0 \\ \hline \kappa-\di \eta & 0 &  0 & 0 & 0 & 0 & 0 & 0\\ \hline \theta-\di \xi & 0 & 0 & 0 & 0 & 0 & 0 & 0 \\ \hline 0 & -\chi &  -\omega & 0 & 0 & 0 & 0 & 0
\end{array}\right).$$
The validity of Assumptions \ref{convf} and \ref{assl} and of condition \eqref{e:invt} can be shown in a way similar to the NLS case.  Let us just check Assumption \ref{a:consc}. We compute, integrating by parts where needed,
\begin{multline*}  (\cf(v), L^*(\mom))=(a,-u)+(b,-v)+(\aaa,-\nabla u)+(\bbb,-\nabla v)+(u,y-\di \aaa)\\ +(v,z-\di \bbb) +(a(a^2+b^2)^{q},-u)+(b(a^2+b^2)^{q},-v)\\=(y,-u)+(z,-v) + (\di\aaa, u)+(\di\bbb, v)+(\nabla \bbb,\nabla \aaa)\\ +(u,y-\di \aaa)+(v,z-\di \bbb)=0.\end{multline*} 


\subsection*{Acknowledgment} The author thanks Yann Brenier and Iv\'an Moyano for inspiring discussions. 

\subsection*{Data availability statement} Data sharing is not applicable to this article as no datasets were
generated or analyzed during the current study.

\appendix \section{An anisotropic Hardy inequality} \label{la1}

\begin{proposition}  Let $\ce$ be a $C^1$-smooth $N$-function such that $\ce$ and $\ce^*$ satisfy the $\Delta_2$-condition. Given $f\in \lk((0,T)\times \Omega;\R^n)$, define the function $$F(t,x):=\frac 1 t\int_0^t f(s,x)\,ds.$$  Then $$F\in \lk((0,T)\times \Omega;\R^n)$$ and \be \|F\|_{\lk((0,T)\times \Omega)}\lesssim \|f\|_{\lk((0,T)\times \Omega)}.\ee \end{proposition}

\begin{proof} It follows from \cite[Lemma 2.3.16]{Gwiazda} that there exist $r>0$ and $p>1$ such that \be \frac {v\cdot  \nabla \ce(v)} {\ce(v)}\geq p \ee for $|v|\geq r$.   
	
	We first observe that \be \label{a09} \ce (sv)\leq \ce (v)s^p, s\in (0,1), v\in \R^n, \ |sv|\ge r.\ee Indeed, fix $s_0\in (0,1)$ with $|s_0 v|\ge r$, and consider the function $$\phi(s)=\log\ce (sv)-\log(\ce (v)s^p), \ s\in [s_0,1].$$ Then $$\phi'(s)=s^{-1}\frac {sv\cdot  \nabla \ce(sv)} {\ce(sv)}-s^{-1}p\geq 0.$$ Since $\phi(1)=0$, we infer $\ce (s_0v)\leq \ce (v)s_0^p$.
	
	It follows from \eqref{a09} that \be \label{a14} \ce (v)\leq s\ce \left(v s^{-1/p}\right), s\in (0,1),\ |v|\ge r.\ee 
	Denote $\ce_s (v):=s\ce \left(v s^{-1/p}\right)$ and $f_s(t,x):=f(st,x), \ s\in(0,1)$. Then \be \label{a165} F(t,x)=\int_0^1 f_s(t,x)\, ds\ee and
	\be \label{a16} \ce (v)\leq \ce_s \left(v\right), \ |v|\ge r.\ee 
	We now claim that \be\label{toint1}\|f_s\|_{\lkss((0,T)\times \Omega)}\le s^{-1/p} \|f\|_{\lk((0,T)\times \Omega)}.\ee Indeed, \begin{multline*} \|f_s\|_{\lkss((0,T)\times \Omega)}= \inf\left\{\lambda>0: \int_{(0,T)\times \Omega}\ce_s\left(\frac {f(ts,x)}{\lambda}\right)\,dx\,dt\leq 1\right\}\\= \inf\left\{\lambda>0: \int_{(0,sT)\times \Omega}\ce_s\left(\frac {f(t,x)}{\lambda}\right)\,dx\,dt\leq s\right\}\\ \le \inf\left\{\lambda>0: \int_{(0,T)\times \Omega}\ce_s\left(\frac {f(t,x)}{\lambda}\right)\,dx\,dt\leq s\right\}\\= \inf\left\{\lambda>0: \int_{(0,T)\times \Omega}\ce\left(\frac {f(t,x)}{\lambda s^{1/p}}\right)\,dx\,dt\leq 1\right\}\\=s^{-1/p} \|f\|_{\lk((0,T)\times \Omega)}. \end{multline*}
	
	Employing convexity of the Luxemburg norm, Proposition \ref{la0}, \eqref{a165} and \eqref{a16}, and integrating \eqref{toint1} w.r.t. $s\in (0,1)$, we conclude that
	\begin{multline*} \|F\|_{\lk((0,T)\times \Omega)}=\left\|\int_0^1 f_s\, ds\right \|_{\lk((0,T)\times \Omega)}\\ \le \int_0^1\left\| f_s\right \|_{\lk((0,T)\times \Omega)}\, ds \le C_r  \int_0^1\left\| f_s\right \|_{\lkss((0,T)\times \Omega)}\, ds\le \frac {p C_r}{p-1}  \|f\|_{\lk((0,T)\times \Omega)}.\end{multline*}
	
\end{proof}

\section{Ballistic optimal transport} \label{bot} Let us briefly describe the link of our dual problem \eqref{o:conc} with the optimal transportation problems. We employ a geometric intuition that did not explictly appear in the previous works. 

Let us start from the following heuristic setting.  Let $\mathfrak M$ be a complete, connected Riemannian manifold. 
Let $X:\mathfrak M\to T\mathfrak M$ be a given vector field on $\mathfrak M$, and let $x_T\in \mathfrak M$ be a prescribed point. 
By a \emph{ballistic}\footnote{The wording is borrowed from \cite{BG19}.} geodesic problem we mean finding the curve\footnote{The subscript $t$ is not a time derivative but just the value at time $t$.}  $\gamma_t\subset \mathfrak M$ that minimizes
\be  \min_{{\dot\gamma_0=X(\gamma_0)},\ \gamma_T=x_T}\frac 12 \int_0^T\langle \dot\gamma_t,\dot\gamma_t \rangle_{\gamma_t}\, dt, \label{e:b1}\ee
i.e., we know the final position and the initial direction of an unknown geodesic. This can be viewed as a mixed boundary condition (``inhomogeneous'' Neumann at $t=0$, Dirichlet at $t=T$). 

Assume that the vector field $X$ that determines the direction of the ``howitzer'' at every point of $\mathfrak M$  is a potential field, i.e., 
$X=\grad \Phi$ for some $\Phi:\mathfrak M \to \R$.  Then \eqref{e:b1} is equivalent to the Hopf-Lax formula
\be \label{hl}  \min_{x_0\in \mathfrak M}\Phi(x_0)+\frac 1 2 dist^2(x_0,x_T),\ee where $dist$ is the Riemannian distance on $\mathfrak M$ (here we are merely interested in the minimizer $x_0$ that determines the starting point of the unknown geodesic and not in the optimal value itself). 

Formula \eqref{hl} makes sense even if we have no manifold structure and $\mathfrak M$ is merely a metric space, not necessarily connected. This is particularly relevant for the metric spaces with ``Riemannian flavour'', see \cite[p. 6]{MS20} for a list of examples of such spaces. Problem \eqref{hl} indeed appears in infinite-dimensional and metric-geometric contexts, including the context of optimal transport.  The generalization to $p>1$, i.e., \be \label{hlp}  \min_{x_0\in \mathfrak M}\Phi(x_0)+\frac 1 p dist^p(x_0,x_T)\ee was studied in \cite{am13} in a metric context. Problems of this kind are called marginal entropy-transport problems in \cite{LMS18}. They are crucial ingredients of de Giorgi's minimizing movement scheme,  also known as the JKO scheme \cite{JKO,AGS}.

Consider now the quadratic Hamilton-Jacobi equation \begin{equation*} \p_t \psi+\frac 1 2 |\nabla \psi|^2=0, \ \psi(0,x)=\psi_0(x), (t,x)\in [0,T]\times \Omega.\end{equation*} Setting $v=\nabla \psi$, we rewrite it in the form \be \label{e:hjv}  \p_t v+\frac 1 2 \nabla \tr (v\otimes v)=0, \ v(0)=\nabla \psi_0. \ee 
Note that \eqref{e:hjv} is  a particular case of our abstract equation \eqref{o:aeuler} with $\cf(v)=v\otimes v$, $L=-\frac 1 2 \nabla \tr$.

Let $\wei\equiv 1$. As explained in \cite[Section 2.2]{V22}, the dual problem \eqref{o:conc} for the quadratic Hamilton-Jacobi equation  may be rewritten (this is formal but can be made rigorous) as \be\label{e:concinth}   -\int_{\Omega} \psi_0 \rho(0)\,dx  -\frac 1 2 \int_0^T\int_{\Omega} \rho^{-1}|q|^2\,dx \, dt \to \sup \ee subject to the constraints \be\label{e:constrinth}  \p_t \rho+\di q=0, \quad \rho(T)=1,\quad \rho\geq 0.\ee 
Rescaling if needed, we can assume $T=1$ and $|\Omega|=1$. Defining the functional $\Psi(\rho):=\int_{\Omega} \psi_0 d\rho$ on the Wasserstein space $\mathcal P(\Omega)$ of probability measures on $\Omega$, multiplying by $-1$ and leveraging the Benamou-Brenier formula \cite{villani03topics,BB2000}, we can  recast \eqref{e:concinth}, \eqref{e:constrinth} in the Hopf-Lax  form \be \label{hl2}  \min_{\rho_0\in \mathcal P(\Omega)}\Psi(\rho_0)+\frac 1 2 W_2^2(\rho_0,\rho_1).\ee Here $W_2$ is the quadratic Wasserstein distance \cite{villani03topics}, and $d\rho_1$ is the Lebesgue measure $dx$ on $\Omega$.  Hence, in this very particular but instructive case, the dual problem \eqref{o:conc} can be viewed as a ballistic problem on the Wasserstein space. 

The observations above can be generalized to the case of the $p$-Wasserstein space (and very likely to the Orlicz-Wasserstein spaces \cite{sturm2011,am13}). The generating Hamilton-Jacobi equation is $$ \p_t v+\frac {p-1} p \nabla \tr \cf(v)=0$$ with \be \cf(v)= |v|^{\frac{2-p}{p-1}} v\otimes v.\label{e:last}\ee We omit the implementation but call attention to the details that the matrix function $\cf$ in 
\eqref{e:last} is $\Lam$-convex but not Loewner convex, and that the resulting Hopf-Lax formula is similar to \eqref{hlp} because the $p$-Wasserstein space is not `Riemannian-like'' unless $p=2$.

\end{document}